\documentclass[12pt,a4paper]{amsart}

\usepackage{amsmath,amssymb,amsfonts,amsthm,latexsym,graphicx,multirow,color}
\usepackage{hyperref}

\newcommand\A{\mathrm{A}}  \newcommand\AGaL{\mathrm{A\Gamma L}}  \newcommand\Aut{\mathrm{Aut}}
 \newcommand\bbF{\mathbb{F}} \newcommand\bfO{\mathbf{O}}
\newcommand\C{\mathrm{C}}  \newcommand\calC{\mathcal{C}} \newcommand\calS{\mathcal{S}}  \newcommand\Cen{\mathbf{C}}  
\newcommand\D{\mathrm{D}}

\newcommand\G{\mathrm{G}}   \newcommand\GaL{\mathrm{\Gamma L}}
  \newcommand\GL{\mathrm{GL}}

\newcommand\M{\mathrm{M}} \newcommand\magma{{\sc Magma} }

\newcommand\Nor{\mathbf{N}}
\newcommand\Out{\mathrm{Out}}
  \newcommand\PGL{\mathrm{PGL}} \newcommand\PGaL{\mathrm{P\Gamma L}}
  \newcommand\POm{\mathrm{P\Omega}} \newcommand\ppd{\mathrm{ppd}} \newcommand\PSL{\mathrm{PSL}}   \newcommand\PSO{\mathrm{PSO}} \newcommand\PSp{\mathrm{PSp}} \newcommand\PSU{\mathrm{PSU}}

 \newcommand\Rad{\mathrm{Rad}}  
 \newcommand\SL{\mathrm{SL}}  \newcommand\Soc{\mathrm{Soc}} \newcommand\Sp{\mathrm{Sp}}    \newcommand\Sy{\mathrm{S}}  

\newcommand\Z{\mathbf{Z}} 

\newtheorem{theorem}{Theorem}[section]
\newtheorem{lemma}[theorem]{Lemma}
\newtheorem{proposition}[theorem]{Proposition}
\newtheorem{corollary}[theorem]{Corollary}

\theoremstyle{definition}
\newtheorem{definition}[theorem]{Definition}

\newtheorem{question}[theorem]{Question}
\newtheorem{hypothesis}[theorem]{Hypothesis}

\begin{document}

\title{Vertex-primitive $s$-arc-transitive digraphs of linear groups}

\author[Giudici]{Michael Giudici}
\address{Department of Mathematics and Statistics\\The University of Western Australia\\Crawley 6009, WA\\Australia}
\email{michael.giudici@uwa.edu.au}

\author[Li]{Cai Heng Li}
\address{Department of Mathematics, Southern University of Science and Technology\\Shenzhen 518055, Guangdong\\P. R. China}
\email{lich@sustc.edu.cn}

\author[Xia]{Binzhou Xia}
\address{School of Mathematics and Statistics\\The University of Melbourne\\Parkville, VIC 3010\\Australia}
\email{binzhoux@unimelb.edu.au}

\maketitle

\begin{abstract}
We study $G$-vertex-primitive and $(G,s)$-arc-transitive digraphs for almost simple groups $G$ with socle $\mathrm{PSL}_n(q)$. We prove that $s\leqslant2$ for such digraphs, which provides the first step in determining an upper bound on $s$ for all the vertex-primitive $s$-arc-transitive digraphs.


\end{abstract}

\section{Introduction}

A \emph{digraph} $\Gamma$ is a pair $(V,\rightarrow)$ with a set $V$ (of vertices) and an antisymmetric irreflexive binary relation $\rightarrow$ on $V$. For a non-negative integer $s$, an \emph{$s$-arc} of $\Gamma$ is a sequence $v_0,v_1,\dots,v_s$ of vertices with $v_i\rightarrow v_{i+1}$ for each $i=0,\dots,s-1$. A $1$-arc is also simply called an \emph{arc}. For a subgroup $G$ of $\Aut(\Gamma)$, we say $\Gamma$ is \emph{$(G,s)$-arc-transitive} if $G$ acts transitively on the set of $s$-arcs of $\Gamma$. An $(\Aut(\Gamma),s)$-arc-transitive digraph $\Gamma$ is said to be \emph{$s$-arc-transitive}. Note that a vertex-transitive $(s+1)$-arc-transitive digraph is necessarily $s$-arc-transitive.  A transitive permutation group $G$ on a set $\Omega$ is said to be \emph{primitive} if $G$ does not preserve any nontrivial partition of $\Omega$. For a subgroup $G$ of $\Aut(\Gamma)$, we say $\Gamma$ is \emph{$G$-vertex-primitive} if $G$ is primitive on the vertex set. An $\Aut(\Gamma)$-vertex-primitive digraph $\Gamma$ is said to be \emph{vertex-primitive}. All digraphs and groups considered in this paper will be finite.

It appears that vertex-primitive $s$-arc-transitive digraphs with large $s$ are very rare. Indeed, the existence of vertex-primitive $2$-arc-transitive digraphs besides directed cycles was only recently determined~\cite{GLX2017} and no vertex-primitive $3$-arc-transitive examples are known. In~\cite{GX2018} the authors asked the following question:

\begin{question}\label{qes1}
Is there an upper bound on $s$ for vertex-primitive $s$-arc-transitive digraphs that are not directed cycles?
\end{question}

A group $G$ is said to be \emph{almost simple} if $G$ has a unique minimal normal subgroup $T$ and $T$ is a nonabelian simple group. These are precisely the groups lying between a nonabelian simple group $T$ and its automorphism group $\Aut(T)$. A systematic investigation of the O'Nan-Scott types of primitive groups has reduced Question~\ref{qes1} to almost simple groups by showing that an upper bound on $s$ for vertex-primitive $s$-arc-transitive digraphs $\Gamma$ with $\Aut(\Gamma)$ almost simple will be an upper bound on $s$ for all vertex-primitive $s$-arc-transitive digraphs~\cite[Corollary~1.6]{GX2018}. This paper provides the first step in determining such an upper bound by studying vertex-primitive $s$-arc-transitive digraphs whose automorphism group is an almost simple linear group. Our main result is as follows.

\begin{theorem}\label{thm1}
Let $\Gamma$ be a $G$-vertex-primitive $(G,s)$-arc-transitive digraph, where $G$ is almost simple with socle $\PSL_n(q)$. Then $s\leqslant2$.
\end{theorem}

We remark that an infinite family of $G$-vertex-primitive $(G,2)$-arc-transitive digraphs with $G=\PSL_3(p^2)$ for each prime $p>3$ such that $p\equiv\pm2\pmod{5}$ was constructed in~\cite{GLX2017}. These digraphs have vertex stabilizer $\A_6$ and arc-stabilizer $\A_5$, and are the only known examples of $G$-vertex-primitive $(G,2)$-arc-transitive digraphs such that $G$ is almost simple. A complete classification of $G$-vertex-primitive $(G,2)$-arc-transitive digraphs for almost simple groups $G$, even for those with $\Soc(G)=\PSL_n(q)$, seems out of reach at this stage, though would be achievable for small values of $n$.

Note that if $\Soc(G)=\PSL_n(q)$ then either $G\leqslant\PGaL_n(q)$ or $G$ has an index $2$ subgroup contained in $\PGaL_n(q)$ and $G$ contains an element that acts on the projective space associated with $G$ by interchanging the set of $1$-spaces and the set of hyperplanes. For any $G$-vertex-primitive $(G,s)$-arc-transitive digraph $\Gamma$, the vertex stabilizer $G_v$ for any vertex $v$ of $\Gamma$ is maximal in $G$. We prove Theorem~\ref{thm1} by analyzing the maximal subgroups of $G$ according to the classes provided by Aschbacher's theorem~\cite{Aschbacher1984}. The classes $\calC_1$, $\calC_2$,\ldots, $\calC_8$ are discussed in Sections~\ref{sec1}--\ref{sec2}, while the remaining class $\calC_9$ is dealt with in Section~\ref{sec3}.
We actually prove that there is no $G$-vertex-primitive $(G,2)$-arc-transitive digraph with $G_v$ from classes $\calC_3$,\ldots,$\calC_6$ (Theorem~\ref{thm3}) though the possibility for an example with $G_v$ from classes $\calC_1$, $\calC_2$, $\calC_7$ or $\calC_8$ remains open. The examples in \cite{GLX2017} have $G_v$ from the class $\calC_9$. At the end of Section~\ref{sec2} we give a proof of Theorem~\ref{thm1}.

\section{Preliminaries}\label{sec5}

\subsection{Notation}

For a group $X$, denote by $\Soc(X)$ the socle of $X$ (that is, the product of all minimal normal subgroups of $X$), $F(G)$ the Fitting subgroup of $G$, $\Rad(X)$ the largest soluble normal subgroup of $X$, and $X^{(\infty)}$ the smallest normal subgroup of $X$ such that $X/X^{(\infty)}$ is soluble.

For a group $X$ and a prime $p$, denote by $\bfO_p(X)$ the largest normal $p$-subgroup of $X$, and $\Omega_p(X)$ the subgroup of $X$ generated by the elements of order $p$ in $X$.

For any integer $n$ and prime number $p$, denote by $n_p$ the $p$-part of $n$ (that is, the largest power of $p$ dividing $n$) and $\pi(n)$ the set of prime divisors of $n$. If $X$ is a group, then $\pi(X):=\pi(|X|)$. The following result is a consequence of the so-called Legendre's formula, which we will use repeatedly in this paper.

\begin{lemma}\label{Factorial}
For any positive integer $n$ and prime $p$ we have $(n!)_p<p^{n/(p-1)}$.
\end{lemma}

Given integers $a\geqslant2$ and $m\geqslant2$, a prime number $r$ is called a \emph{primitive prime divisor} of the pair $(a,m)$ if $r$ divides $a^m-1$ but does not divide $a^i-1$ for any positive integer $i<m$. By an elegant theorem of Zsigmondy (see for example~\cite[Theorem IX.8.3]{Blackburn1982}), $(a,m)$ always has a primitive prime divisor except when $(a,m)=(2,6)$ or $a+1$ is a power of $2$ and $m=2$. Denote the set of primitive prime divisors of $(a,m)$ by $\ppd(a,m)$ if $(a,m)\neq(2,6)$, and set $\ppd(2,6)=\{7\}$. Note that for each $r\in\ppd(a,m)$ Fermat's Little Theorem implies that $r\equiv1\pmod{m}$ and so $r>m$.

\subsection{Group factorizations}

An expression of a group $G$ as the product of two subgroups $H$ and $K$ of $G$ is called a \emph{factorization} of $G$, where $H$ and $K$ are called \emph{factors}. The following lemma lists several equivalent conditions for a group factorization, whose proof is fairly easy and so is omitted.

\begin{lemma}\label{Factorization}
Let $H$ and $K$ be subgroups of $G$. Then the following are equivalent:
\begin{itemize}
\item[(a)] $G=HK$.
\item[(b)] $G=KH$.
\item[(c)] $G=(x^{-1}Hx)(y^{-1}Ky)$ for any $x,y\in G$.
\item[(d)] $|H\cap K||G|=|H||K|$.
\item[(e)] $H$ acts transitively by right multiplication on the set of right cosets of $K$ in $G$.
\item[(f)] $K$ acts transitively by right multiplication on the set of right cosets of $H$ in $G$.
\end{itemize}
\end{lemma}

We give some lemmas below concerning factorizations of almost simple groups, which are not only needed later but also of interest in their own right.

\begin{lemma}\label{lem1}
Suppose $G=\A_n$ or $\Sy_n$ acts naturally on a set $\Omega$ of size $n\geqslant2$ and $G=HK$ with subgroups $H$ and $K$ of $G$. Then at least one of $H$ or $K$ is transitive on $\Omega$.
\end{lemma}

\begin{proof}
Suppose for a contradiction that neither $H$ nor $K$ is transitive on $\Omega$. Then $H$ stabilizes a subset $\Delta$ of $\Omega$ with $|\Delta|\leqslant n/2$ and $K$ stabilizes a subset $\Lambda$ of $\Omega$ with $|\Lambda|\leqslant n/2$. Without loss of generality assume $|\Delta|\leqslant|\Lambda|$. Then as $|\Delta|\leqslant n/2\leqslant|\Omega\setminus\Lambda|$, there exist subsets $\Delta_1$ and $\Delta_2$ of $\Omega$ such that $|\Delta_1|=|\Delta_2|=|\Delta|$,
\begin{equation}\label{eq1}
\Delta_1\subseteq\Lambda\quad\text{and}\quad\Delta_2\subseteq\Omega\setminus\Lambda.
\end{equation}
Since $G=HK$ and $H\leqslant G_\Delta$, we have $G=G_\Delta K$, and so Lemma~\ref{Factorization}(f) implies that $K$ is transitive on the set of right cosets of $G_\Delta$ in $G$. Consequently, $K$ is transitive on the set of subsets of $\Omega$ of size $|\Delta|$. In particular, there exists $g\in K$ such that $\Delta_1^g=\Delta_2$. However, as $g\in K$ stabilizes $\Lambda$, this contradicts~\eqref{eq1}.
\end{proof}

Factorizations of almost simple groups with socle $\A_n$ have been classified in~\cite[Theorem~D]{LPS1990}, from which one may derive the following:

\begin{lemma}\label{lem4}
Suppose $G=\A_n$ or $\Sy_n$ acts naturally on a set $\Omega$ of size $n\geqslant2$ and $G=HK$ with subgroups $H$ and $K$ of $G$. If both $H$ and $K$ are transitive on $\Omega$, then one of the following holds:
\begin{itemize}
\item[(a)] At least one of $H$ or $K$ contains $\A_n$.
\item[(b)] $n=6$, and interchanging $H$ and $K$ if necessary, $\PSL_2(5)\leqslant H\leqslant\PGL_2(5)$ and $K\leqslant\Sy_3\wr\Sy_2$.
\end{itemize}
\end{lemma}

The next lemma is also based on the classification of factorizations of almost simple groups with socle $\A_n$.

\begin{lemma}\label{lem10}
Suppose $G=\A_n$ or $\Sy_n$ with $n\geqslant7$ and $G=HK$ with subgroups $H$ and $K$ of $G$. If $H$ and $K$ have the same set of insoluble composition factors, then both $H$ and $K$ contain $\A_n$.
\end{lemma}

\begin{proof}
Let $G$ act naturally on a set $\Omega$ of size $n$. If either $H$ or $K$ contains $\A_n$, then the other also contains $\A_n$ since $H$ and $K$ have the same set of insoluble composition factors. To complete the proof we suppose that neither $H$ nor $K$ contains $\A_n$. Then by~\cite[Theorem~D]{LPS1990}, interchanging $H$ and $K$ if necessary, $\A_{n-k}\leqslant H\leqslant\Sy_{n-k}\times\Sy_k$ and $K$ is $k$-homogeneous for some $1\leqslant k\leqslant5$. The $k$-homogeneous but not $k$-transitive permutation groups are classified in~\cite{Kantor1972}, while the $k$-transitive permutation groups with $k\geqslant2$ are well-known (see for example~\cite{Cameron1999}). This gives us a list of all the $k$-homogeneous permutation groups.

First assume that $k=1$. Then $\A_{n-1}$ is an insoluble composition factor of $H$, and hence is a composition factor of $K\cap\A_n$. As a consequence, $|\A_{n-1}|$ divides $|K\cap\A_n|$. This implies that $K\cap\A_n\cong\A_{n-1}$. Since $n\geqslant7$, it follows that $K\cap\A_n$ fixes a unique point of $\Omega$, contradicting the condition that $K$ is transitive.

Next assume that $k=2$. Then $\A_{n-2}$ is an insoluble composition factor of $H$, and hence is a composition factor of $K\cap\A_n$. Moreover, $K$ is $2$-homogeneous. However, checking the list of $2$-homogeneous permutation groups we see that there is no $2$-homogeneous permutation group $K$ of degree $n$ with a composition factor isomorphic to $\A_{n-2}$, a contradiction.

Now assume that $k=3$ or $4$. If $n-k\leqslant4$, then $H$ is soluble and so is $K$. Inspecting the $k$-homogeneous permutation groups of degree $n$ for $3\leqslant k\leqslant4$ and $7\leqslant n\leqslant k+4$ we see that this is not possible. Therefore, $n-k\geqslant5$ so that $\A_{n-k}$ is an insoluble composition factor of $H$ and hence $K$. However, checking the list of $k$-homogeneous permutation groups we see that there is no $k$-homogeneous permutation group $K$ of degree $n$ with a composition factor isomorphic to $\A_{n-k}$ for $3\leqslant k\leqslant4$, a contradiction.

Finally assume that $k=5$. Then according to the list of $5$-homogeneous permutation groups, either $\PSL_2(7)\leqslant K\leqslant\PGaL_2(7)$, or $K$ is one of the groups:
\[
\AGaL_1(7),\ \PSL_2(8),\ \PGaL_2(8),\ M_{12},\ M_{24}.
\]
However, as $\A_{n-5}\leqslant H\leqslant\Sy_{n-5}\times\Sy_5$, it is not possible for $H$ and $K$ to have the same set of insoluble composition factors.
\end{proof}

The next three lemmas on factorizations of almost simple groups are obtained by checking \cite[Tables~1--3]{LPS1990} as a consequence of~\cite[Theorem~A]{LPS1990} and~\cite{LPS1996}.

\begin{lemma}\label{lem20}
Let $G$ be an almost simple group with socle $L=\PSL_n(q)$, where $q=p^f$ with prime $p$. If $G=HK$ with subgroups $H$ and $K$ of $G$ such that $r\in\pi(H)\cap\pi(K)$ for some $r\in\ppd(p,nf)$, then one of the following holds:
\begin{itemize}
\item[(a)] At least one of $H$ or $K$ contains $L$.
\item[(b)] $n=2$ and $q=9$.
\item[(c)] $n=6$ and $q=2$.
\end{itemize}
\end{lemma}

\begin{lemma}\label{lem19}
Let $G$ be an almost simple group with socle $L=\PSL_n(q)$. If $G=HK$ with subgroups $H$ and $K$ of $G$ such that $\pi(G)\setminus(\pi(q-1)\cup\pi(\Out(L))\subseteq\pi(H)\cap\pi(K)$, then one of the following holds:
\begin{itemize}
\item[(a)] At least one of $H$ or $K$ contains $L$.
\item[(b)] $n=2$ and $q=8$.
\item[(c)] $n=2$ and $q=9$.
\end{itemize}
\end{lemma}

\begin{lemma}\label{lem24}
Let $G$ be an almost simple group with socle $L=\PSp_{2m}(p)$, where $m\geqslant1$ and $p$ is prime. If $G=HK$ with subgroups $H$ and $K$ of $G$ such that $\pi(G)\setminus\pi(p(p-1))\subseteq\pi(H)\cap\pi(K)$, then interchanging $H$ and $K$ if necessary, one of the following holds:
\begin{itemize}
\item[(a)] At least one of $H$ or $K$ contains $L$.
\item[(b)] $m=1$, $p$ is a Mersenne prime, $H\cap L=\D_{p+1}$ and $K\cap L=\C_p\rtimes\C_{(p-1)/2}$.
\item[(c)] $m=1$, $p=7$, $H\cap L=\C_7\rtimes\C_3$ and $K\cap L=\Sy_4$.
\end{itemize}
\end{lemma}

We close this subsection with two results on factorizations of linear groups.

\begin{lemma}\label{lem21}
Let $V$ be a vector space and $G\leqslant\GL(V)$ such that $G$ acts transitively on the set of subspaces of $V$ of any fixed dimension. Suppose $G=HK$ with subgroups $H$ and $K$ stabilizing subspaces $U$ and $W$ of $V$, respectively. Then either $U=0$ or $V$, or $W=0$ or $V$.
\end{lemma}

\begin{proof}
Suppose on the contrary that $0<U<V$ and $0<W<V$. Without loss of generality, assume that $\dim(U)\leqslant\dim(W)$. Since $H$ stabilizes $U$ and $G=HK$ acts transitively on the set of subspaces of $V$ of dimension $\dim(U)$, we conclude that $K$ acts transitively on the set of subspaces of $V$ of dimension $\dim(U)$. Take $U_1$ and $U_2$ to be subspaces of $V$ of dimension $\dim(U)$ such that $U_1\leqslant W$ and $U_2\nleqslant W$. Then since $K$ stabilizes $W$, there is no element of $K$ mapping $U_1$ to $U_2$, a contradiction.
\end{proof}

\begin{lemma}\label{lem22}
Let $V$ be a vector space, $U$ be a nontrivial proper subspace of $V$ and $G\leqslant\GL(V)$ such that $G$ stabilizes $U$ and acts transitively on the set of complements of $U$ in $V$. Suppose $G=HK$ with subgroups $H$ and $K$ such that $H$ stabilizes a complement of $U$ in $V$ and $K$ stabilizes a subspace $W$ of $V$. Then either $W\leqslant U$ or $W+U=V$.
\end{lemma}

\begin{proof}
Suppose for a contradiction that $W\nleqslant U$ and $W+U\neq V$, which means $U<W+U<V$. Extend a basis $e_1,\dots,e_r$ of $U$ to a basis $e_1,\dots,e_r,a_1,\dots,a_s$ of $W+U$ and then a basis $e_1,\dots,e_r,a_1,\dots,a_s,b_1,\dots,b_t$ of $V$. Let
\[
W_1=\langle a_1,a_2,\dots,a_s,b_1,b_2,\dots,b_t\rangle\quad\text{and}\quad W_2=\langle e_1+a_1,a_2,\dots,a_s,b_1,b_2,\dots,b_t\rangle.
\]
Then $W_1$ and $W_2$ are both complements of $U$ in $V$. Since $H$ stabilizes a complement of $U$ in $V$ and $G=HK$ acts transitively on the set of complements of $U$ in $V$, we deduce that $K$ acts transitively on the set of complements of $U$ in $V$. In particular, there exists $g\in K$ such that $W_1^g=W_2$. Since $K$ stabilizes $W$ and $U$, we have $(W+U)^g=W+U$. Hence $(W_1\cap(W+U))^g=W_1^g\cap(W+U)^g=W_2\cap(W+U)$. However, $W_1\cap(W+U)=\langle a_1,a_2,\dots,a_s\rangle$ and $W_2\cap(W+U)=\langle a_2,\dots,a_s\rangle$, so that $\dim(W_1\cap(W+U))=s$ and $\dim(W_2\cap(W+U))=s-1$, a contradiction.
\end{proof}

\subsection{$s$-Arc-transitive digraphs}

We say a group $G$ \emph{acts} on a digraph $\Gamma$ if $G$ acts on the vertex set of $\Gamma$ with image in $\Aut(\Gamma)$. For vertices $v_1,\dots,v_i$ of $\Gamma$ let $G_{v_1\dots v_i}$ denote the subgroup of $G$ that fixes each of $v_1,\dots,v_i$. The following two lemmas slightly extend~\cite[Lemma~2.2]{GX2018} and~\cite[Corollary~2.11]{GX2018} by allowing the action of $G$ on the vertex set to be unfaithful. Their proof is along the same lines as those in~\cite{GX2018}, so are omitted.

\begin{lemma}\label{lem15}
Let $\Gamma$ be a digraph, and $v_0\rightarrow v_1\rightarrow\dots\rightarrow v_{s-1}\rightarrow v_s$ be an $s$-arc of $\Gamma$ with $s\geqslant2$. Suppose that $G$ acts arc-transitively on $\Gamma$. Then $G$ acts $s$-arc-transitively on $\Gamma$ if and only if $G_{v_1\dots v_i}=G_{v_0v_1\dots v_i}G_{v_1\dots v_iv_{i+1}}$ for each $i$ in $\{1,\dots,s-1\}$.
\end{lemma}


\begin{lemma}\label{lem7}
Let $\Gamma$ be a digraph, $G$ be a group acting $s$-arc-transitively on $\Gamma$ with $s\geqslant2$, and $L$ be a normal subgroup of $G$. If $L$ acts transitively on the vertex set of $\Gamma$, then $L$ acts $(s-1)$-arc-transitively on $\Gamma$.
\end{lemma}

A digraph $(V,\rightarrow)$ is said to be \emph{$k$-regular} if both the set $\{u\in V\mid u\rightarrow v\}$ of in-neighbors of $v$ and the set $\{w\in V\mid v\rightarrow w\}$ of out-neighbors of $v$ have size $k$ for all $v\in V$.

\begin{lemma}\label{lem5}
For any vertex-primitive arc-transitive digraph $\Gamma$, either $\Gamma$ is a directed cycle of prime length or $\Gamma$ has valency at least $3$.
\end{lemma}

\begin{proof}
Suppose $\Gamma$ is a $G$-arc-transitive digraph such that $\Gamma$ is not a directed cycle of prime length. Then $\Gamma$ has valency at least $2$. If $\Gamma$ has valency $2$, then by~\cite[Theorem~5]{Neumann1973}, $G$ is a dihedral group of order twice a prime. However, in this case all suborbits of $G$ are self-paired, contradicting $\Gamma$ being a digraph.
\end{proof}

We close this subsection with an observation that will be used repeatedly throughout the paper.

\begin{lemma}\label{lem26}
Let $\Gamma$ be a connected $G$-arc-transitive digraph with arc $v\rightarrow w$. Let $g\in G$ such that $v^g=w$. Then each nontrivial normal subgroup of $G_v$ is not normalized by $g$.
\end{lemma}

\begin{proof}
Suppose that $N$ is a nontrivial normal subgroup of $G_v$ that is normalized by $g$. Then $N$ is normal in $\langle G_v,g\rangle$. Since $\Gamma$ is connected, $\langle G_v,g\rangle=G$. Hence $N$ is normal in $G$, which implies that $N\leqslant\bigcap_{h\in G}G_v^h$. However, $G$ acts faithfully on the vertex set of $\Gamma$ and so $\bigcap_{h\in G}G_v^h=1$, a contradiction.
\end{proof}

\subsection{{Subgroup structure}}

The maximal subgroups of almost simple groups with socle $\PSL_n(q)$ are divided into nine classes $\calC_1$, $\calC_2$, \ldots, $\calC_9$ by Aschbacher's theorem~\cite{Aschbacher1984}. The maximal subgroups in classes $\calC_1$--$\calC_8$ are described in~\cite[Chapter~4]{KL1990} and summarized in~\cite[Table~3.5.A]{KL1990}. Each class is defined by providing some geometric structure that it preserves. The maximal subgroups in class $\calC_9$ (denoted by $\calS$ in~\cite{KL1990}) arise from irreducible representations of quasisimple groups. We call a maximal subgroup a \emph{$\calC_i$-subgroup} if it lies in class $\calC_i$, where $i\in\{1,2,\dots,9\}$.

\subsection{Computational methods}

We will use \magma~\cite{magma} to do computations in some relatively small groups, mainly to search for group factorizations with certain properties. By virtue of part~(c) of Lemma~\ref{Factorization}, we may only consider one representative in a conjugacy class of subgroups as a potential factor of a group factorization. Given a group $G$ and its subgroups $H$ and $K$, to inspect whether $G=HK$ holds we only need to compute the orders of $G$, $H$, $K$ and $H\cap K$ by part~(d) of Lemma~\ref{Factorization}.

Let us illustrate this with an example of computation in the proof of Lemma~\ref{lem14}, where we need to show that for the $\calC_2$-subgroup $G$ of type $\GL_2(3)\wr\Sy_6$ in $\SL_{12}(3)$ there is no homogeneous factorization $G=HK$ with $|H|$ divisible by $2^{15}\cdot3^6\cdot5$ and $H/\Rad(H)\cong K/\Rad(K)\cong\Sy_5$. Note that the group $G$ can be accessed from the list produced by the \magma command \textsf{ClassicalMaximals("L",12,3:classes:=\{2\})}. Create a list of subgroups representing all the conjugacy classes of subgroups of $G$ with order divisible by $2^{15}\cdot3^6\cdot5$. For each pair $(H,K)$ of subgroups from this list, calculate $|H|$, $|K|$ and $|H\cap K|$. Then all those pairs $(H,K)$ with $|H\cap K||G|=|H||K|$ give precisely the factorizations $G=HK$, up to conjugacy of $H$ and $K$ in $G$, such that both $|H|$ and $|K|$ are divisible by $2^{15}\cdot3^6\cdot5$. It turns out that there is no such pair $(H,K)$ satisfying $H\cong K$ and $H/\Rad(H)\cong K/\Rad(K)\cong\Sy_5$. This shows that there is no homogeneous factorization $G=HK$ with $|H|$ divisible by $2^{15}\cdot3^6\cdot5$ and $H/\Rad(H)\cong K/\Rad(K)\cong\Sy_5$.

\section{Homogeneous factorizations}\label{sec3}

From Lemma~\ref{lem15} we see that the $s$-arc-transitivity of digraphs can be characterized by group factorizations. If a group $G$ acts $s$-arc-transitively on a digraph $\Gamma$ with $s\geqslant2$ and $v_0\rightarrow v_1\rightarrow\dots\rightarrow v_{s-1}\rightarrow v_s$ is an $s$-arc of $\Gamma$, then Lemma~\ref{lem15} implies that $G_{v_1\dots v_i}=G_{v_0v_1\dots v_i}G_{v_1\dots v_iv_{i+1}}$ for all $1\leqslant i\leqslant s-1$. In addition, since $G$ acts $i$-arc-transitively on $\Gamma$, the two factors $G_{v_0v_1\dots v_i}$ and $G_{v_1\dots v_iv_{i+1}}$ are conjugate in $G$ and hence isomorphic. This motivates the following definition:

\begin{definition}
A factorization $G=AB$ is called a \emph{homogeneous factorization} of $G$ if $A\cong B$.
\end{definition}

The next lemma shows that the orders of the factors of a homogeneous factorization are necessarily divisible by all the prime factors of the factorized group.

\begin{lemma}\label{lem2}
For any homogeneous factorization $G=AB$ we have $\pi(A)=\pi(B)=\pi(G)$.
\end{lemma}

\begin{proof}
Due to $G=AB$, Lemma~\ref{Factorization}(d) implies that $|G|$ divides $|A||B|$. Moreover, $|A|=|B|$ as $A\cong B$. Hence $|G|$ divides $|A|^2=|B|^2$, which implies that $\pi(G)\subseteq\pi(A)=\pi(B)$. Since $A$ and $B$ are both subgroups of $G$, we also have the reverse containment and so $\pi(A)=\pi(B)=\pi(G)$.
\end{proof}

The following proposition determines the homogeneous factorizations of almost simple groups.

\begin{proposition}\label{Homogeneous}
Let $G$ be an almost simple group with socle $T$. Suppose $G=AB$ with isomorphic subgroups $A$ and $B$ of $G$. Then one of the following holds:
\begin{itemize}
\item[(a)] Both $A$ and $B$ contain $T$.
\item[(b)] $A$ and $B$ are almost simple groups with socles both isomorphic to $S$, where $(T,S)$ lies in Table~$\ref{tab1}$.
\end{itemize}
\end{proposition}

\begin{table}[htbp]
\caption{The pair $(T,S)$ in Proposition~\ref{Homogeneous}(b)}\label{tab1}
\centering
\begin{tabular}{|l|l|l|}
\hline
row & $T$ & $S$\\
\hline
1 & $\A_6$ & $\A_5$\\
2 & $\M_{12}$ & $\M_{11}$\\
3 & $\Sp_4(2^f)$, $f\geqslant2$ & $\Sp_2(4^f)$\\
4 & $\POm_8^+(q)$ & $\Omega_7(q)$\\
\hline
\end{tabular}
\end{table}

\begin{proof}
First assume that at least one of $A$ or $B$, say $A$, contains $T$. Then $A$ is an almost simple group with socle $T$. Since $A\cong B$, $B$ is an almost simple group with socle isomorphic to $T$. As $B\cap T$ is normal in $B$, we thereby derive that either $B\cap T=1$ or $B\cap T\geqslant\Soc(B)\cong T$. If $B\cap T=1$, then $B\cong BT/T\leqslant G/T$, which is not possible since $G/T$ is soluble by the Schreier Conjecture. Thus $B\cap T\geqslant\Soc(B)\cong T$. This yields $T\leqslant B$, so both $A$ and $B$ contain $T$, as described in part~(a) of the proposition.

Next assume that neither $A$ nor $B$ contains $T$. According to Lemma~\ref{lem2} we have $\pi(G)=\pi(A)=\pi(B)$, and so $\pi(T)\subseteq\pi(A)=\pi(B)$. Hence by~\cite[Theorem~1.1]{BP1998}, $(T,G,A\cap T,B\cap T)$ lies in Table~I of~\cite{BP1998}. Then we conclude that $A\cong B$ is almost simple with socle $S$ such that $(T,S)$ lies in Table~\ref{tab1}.
\end{proof}

\begin{corollary}\label{AS}
Let $\Gamma$ be a connected $(G,2)$-arc-transitive digraph, and let $u\rightarrow v$ be an arc of $\Gamma$. Suppose that $G_v$ is almost simple. Then $G_{uv}$ is almost simple with $(\Soc(G_v),\Soc(G_{uv}))=(T,S)$ in Table~$\ref{tab1}$, and $\Gamma$ is not $(G,3)$-arc-transitive.
\end{corollary}

\begin{proof}
Since $\Gamma$ is connected and $G$-arc-transitive, there exists $g\in G$ such that $u^g=v$ and $\langle G_v,g\rangle=G$. Let $w=v^g$. Then $u\rightarrow v\rightarrow w$ is a $2$-arc of $\Gamma$, and so by Lemma~\ref{lem15}, $G_v=G_{uv}G_{vw}$ since $\Gamma$ is $(G,2)$-arc-transitive. Moreover, $G_{uv}^g=G_{u^g v^g}=G_{vw}$. Thus, appealing to Proposition~\ref{Homogeneous} we obtain that either both $G_{uv}$ and $G_{vw}$ contain $\Soc(G_v)$, or $G_{uv}$ is almost simple with $(\Soc(G_v),\Soc(G_{uv}))=(T,S)$ in Table~\ref{tab1}.

Suppose that both $G_{uv}$ and $G_{vw}$ contain $\Soc(G_v)$. Then $\Soc(G_{uv})=\Soc(G_{vw})=\Soc(G_v)$, and so
\[
\Soc(G_v)^g=\Soc(G_{uv})^g=\Soc(G_{uv}^g)=\Soc(G_{vw})=\Soc(G_v).
\]
Since $\Soc(G_v)$ is normal in $G_v$, this contradicts Lemma~\ref{lem26}. Thus $G_{uv}$ is almost simple and $(\Soc(G_v),\Soc(G_{uv}))=(T,S)$ lies in Table~\ref{tab1}.

Suppose that $\Gamma$ is $(G,3)$-arc-transitive. Then  by Lemma~\ref{lem15}, $G_{uv}$ has a homogeneous factorization with factors isomorphic to $G_{uvw}$. Now $\Soc(G_{uv})=S$ as given in Table~\ref{tab1}. Since none of the possibilities for $S$ also appear as possibilities for $T$ in Table~\ref{tab1}, Proposition~\ref{Homogeneous} implies that $S=\Soc(G_{uv})\leqslant G_{uvw}$. It follows that $|G_v|/|G_{uv}|=|G_{uv}|/|G_{uvw}|$ divides $|\Out(S)|$. However, as $(\Soc(G_v),\Soc(G_{uv}))=(T,S)$ lies in Table~\ref{tab1}, we see that $|G_v|/|G_{uv}|$ cannot divide $|\Out(S)|$. This contradiction shows that $\Gamma$ is not $(G,3)$-arc-transitive, completing the proof.
\end{proof}

In the next lemma we study homogeneous factorizations of wreath products.

\begin{lemma}\label{lem3}
Let $R\wr\Sy_k$ be a wreath product with base group $M=R_1\times\dots\times R_k$, where $R_1\cong\cdots\cong R_k\cong R$, and $T\wr\Sy_k\leqslant G\leqslant R\wr\Sy_k$ with $T\leqslant R$. Suppose $G=AB$ is a homogeneous factorization of $G$ such that $A$ is transitive on $\{R_1,\dots,R_k\}$. Then with $\varphi_i(A\cap M)$ being the projection of $A\cap M$ to $R_i$, we have $\varphi_1(A\cap M)\cong\cdots\cong\varphi_k(A\cap M)$ and $\pi(T)\subseteq\pi(\varphi_1(A\cap M))$.
\end{lemma}

\begin{proof}
Since $A$ is transitive on $\{R_1,\dots,R_k\}$, we have $\varphi_1(A\cap M)\cong\cdots\cong\varphi_k(A\cap M)$. Consequently, $|A|=|AM/M||A\cap M|$ divides
\[
|AM/M||\varphi_1(A\cap M)|\cdots|\varphi_k(A\cap M)|=|AM/M||\varphi_1(A\cap M)|^k.
\]
Since $G=AB$ and $A\cong B$, we see that $|G|$ divides $|A|^2$. Moreover, $|G|$ is divisible by $|T|^k|\Sy_k|$. We conclude that $|T|^k|\Sy_k|$ divides $|AM/M|^2|\varphi_1(A\cap M)|^{2k}$, and so
\begin{equation}\label{eq2}
\text{$|T|^k$ divides $|\Sy_k||\varphi_1(A\cap M)|^{2k}$}
\end{equation}
as $AM/M\leqslant\Sy_k$. For any $p\in\pi(T)$, it follows from~\eqref{eq2} and Lemma~\ref{Factorial} that
\[
p^k\leqslant(k!)_p|\varphi_1(A\cap M)|_p^{2k}<p^{k/(p-1)}|\varphi_1(A\cap M)|_p^{2k},
\]
which yields
\[
|\varphi_1(A\cap M)|_p>\frac{p^{1/2}}{p^{1/(2p-2)}}\geqslant1.
\]
Thus, $\pi(T)\subseteq\pi(\varphi_1(A\cap M))$.
\end{proof}

We close this section with a technical lemma on homogeneous factorizations.

\begin{lemma}\label{lem9}
Let $G=AB$ be a homogeneous factorization, and $N$ be a normal subgroup of $G$ such that $G/N=\Sy_n$ for some positive integer $n$. Suppose that $AN/N$ and $BN/N$ are transitive subgroups of $\Sy_n$. Then one of the following holds:
\begin{itemize}
\item[(a)] $AN/N=BN/N=\Sy_n$.
\item[(b)] $n\geqslant4$, and $N$ has a section isomorphic to $\A_n$.
\item[(c)] $n=3$, and $N$ has a section isomorphic to $\Sy_3$.
\item[(d)] $n=6$, and $N$ has a section isomorphic to $\A_5$.
\end{itemize}
\end{lemma}

\begin{proof}
From $G=AB$ and $G/N=\Sy_n$ we deduce $\Sy_n=(AN/N)(BN/N)$. Then by virtue of Lemma~\ref{lem4}, interchanging $A$ and $B$ if necessary, we only need to deal with the following cases:
\begin{itemize}
\item[(i)] $n\geqslant3$, $AN/N=\Sy_n$ and $BN/N=\A_n$.
\item[(ii)] $n\geqslant3$, $AN/N\geqslant\A_n$ and $BN/N\ngeqslant\A_n$.
\item[(iii)] $n=6$, $\PSL_2(5)\leqslant AN/N\leqslant\PGL_2(5)$ and $BN/N\leqslant\Sy_3\wr\Sy_2$.
\end{itemize}
Since $B\cap N$ is a normal subgroup of $B$ and $A\cong B$, we know that $A$ has a normal subgroup $M$ isomorphic to $B\cap N$ such that $A/M\cong B/(B\cap N)$. It follows that $M$ has index $|B|/|B\cap N|=|BN/N|$ in $A$, and so $MN/N$ is a normal subgroup of $AN/N$ of index dividing $|BN/N|$. From the isomorphisms $M/(M\cap N)\cong MN/N$ and $(A/(M\cap N))/(M/(M\cap N))\cong A/M\cong B/(B\cap N)\cong BN/N$ we deduce that
\[
A/(M\cap N)\cong(MN/N).(BN/N).
\]
Moreover, $A/(M\cap N)$ has a factor group isomorphic to $AN/N$ since
\[
(A/(M\cap N))/((A\cap N)/(M\cap N))\cong A/(A\cap N)\cong AN/N.
\]
Furthermore, since $B\cap N\cong M$ and $M/(M\cap N)\cong MN/N$, it follows that every section of $MN/N$ is isomorphic to a section of $N$.

First assume that~(i) occurs. If $n=3$, then since the only normal subgroup of $AN/N=\Sy_3$ of index dividing $|BN/N|=3$ is $\Sy_3$, we deduce that $MN/N=\Sy_3$ and so $N$ has a section isomorphic to $\Sy_3$. If $n\geqslant5$, then since a normal subgroup of $AN/N=\Sy_n$ either contains $\A_n$ or has index $|\Sy_n|$ and $|BN/N|<|\Sy_n|$, we deduce that $MN/N\geqslant\A_n$, which implies that $N$ has a section isomorphic to $\A_n$. Now assume that $n=4$. Since the only normal subgroups of $AN/N=\Sy_4$ of index dividing $|BN/N|=12$ are $\C_2^2$, $\A_4$ and $\Sy_4$, it follows that either $MN/N=\C_2^2$ or $MN/N\geqslant\A_4$. Suppose $MN/N=\C_2^2$. Then $A/(M\cap N)\cong(MN/N).(BN/N)$ has form $\C_2^2.\A_4$. However, no group of the form $\C_2^2.\A_4$ has a factor group isomorphic to $AN/N=\Sy_4$, a contradiction. Consequently, $MN/N\geqslant\A_4$. Hence again $N$ has a section isomorphic to $\A_4$.

Next assume that~(ii) occurs. If $n=4$, then $BN/N\leqslant\D_8$, and since a normal subgroup of $\A_4$ or $\Sy_4$ of index dividing $|\D_8|=8$ must contain $\A_4$, we deduce that $MN/N\geqslant\A_4$ and hence $N$ has a section isomorphic to $\A_4$. Now assume that $n\geqslant5$. Then $|BN/N|<|\A_n|$ and a normal subgroup of $AN/N$ either contains $\A_n$ or has index at least $|\A_n|$. We thus deduce that $MN/N\geqslant\A_n$ and so $N$ has a section isomorphic to $\A_n$.

Finally, assume that~(iii) occurs. Since a normal subgroup of $\PSL_2(5)$ or $\PGL_2(5)$ of index dividing $|\Sy_3\wr\Sy_2|=72$ must contain $\PSL_2(5)$, we have $MN/N\geqslant\PSL_2(5)$. This implies that $N$ has a section isomorphic to $\PSL_2(5)\cong\A_5$.
\end{proof}

\section{$\calC_1$ and $\calC_2$-subgroups}\label{sec1}

\begin{hypothesis}\label{hyp1}
Let $\Gamma$ be a $G$-vertex-primitive $(G,2)$-arc-transitive digraph of valency at least $3$, where $G$ is almost simple with socle $L=\PSL_n(q)$ and $q=p^f$ for some prime $p$. Take an arc $u\rightarrow v$ of $\Gamma$. Let $\overline{g}$ be an element of $L$ such that $u^{\overline{g}}=v$ and let $w=v^{\overline{g}}$. Then $u\rightarrow v\rightarrow w$ is a $2$-arc in $\Gamma$. Let $X=\SL_n(q)$ acting naturally on $V=\bbF_q^n$, $\varphi$ be the projection from $X$ to $L$, and $g$ be a preimage of $\overline{g}$ under $\varphi$.
\end{hypothesis}

Under Hypothesis~\ref{hyp1}, if in addition $\Gamma$ is $(G,3)$-arc-transitive, then Lemma~\ref{lem7} asserts that $\Gamma$ is $(L,2)$-arc-transitive so that $L_v=L_{uv}L_{vw}$ with $L_{uv}^{\overline{g}}=L_{vw}$, and the composition of $\varphi$ and the action of $L$ on the vertex set of $\Gamma$ gives an action of $X$ on the vertex set of $\Gamma$, under which $X$ acts $2$-arc-transitively on $\Gamma$. This gives $\langle X_v,g\rangle=X$ and $X_v=X_{uv}X_{vw}$ with $X_{uv}^g=X_{vw}$.

\begin{lemma}\label{lem16}
Suppose that Hypothesis~$\ref{hyp1}$ holds and $G_v$ is a $\calC_1$-subgroup of $G$. Then $G\nleqslant\PGaL_n(q)$, $G_v$ does not stabilize a nontrivial proper subspace of $V$, and $\Gamma$ is not $(G,3)$-arc-transitive.
\end{lemma}

\begin{proof}
As $G_v$ is a $\calC_1$-subgroup of $G$, one of the following holds:
\begin{itemize}
\item[(i)] $X_v$ is the stabilizer of a nontrivial proper subspace $W$ of $V$;
\item[(ii)] $G\nleqslant\PGaL_n(q)$, $X_v$ is the stabilizer of two nontrivial proper subspaces $U$ and $W$ of $V$ such that $V=U\oplus W$ and $\dim(U)<n/2$;
\item[(iii)] $G\nleqslant\PGaL_n(q)$, $X_v$ is the stabilizer of two subspaces $U$ and $W$ of $V$ such that $U<W$, $\dim(U)=m>0$ and $\dim(W)=n-m$.
\end{itemize}

Suppose that~(i) holds. Let $e_1,\dots,e_s$ be a basis of $W\cap W^g$. Then there exist $a_1,\dots,a_t$ and $b_1,\dots,b_t$ in $V$ such that $e_1,\dots,e_s,a_1,\dots,a_t$ is a basis of $W$ and $e_1,\dots,e_s,b_1,\dots,b_t$ is a basis of $W^g$. It follows that $e_1,\dots,e_s,a_1,\dots,a_t,b_1,\dots,b_t$ is a basis of $W+W^g$. Take $g_1$ to be an element of $X$ such that $\langle e_i\rangle^{g_1}=\langle e_i\rangle$ for $1\leqslant i\leqslant s$ and $\langle a_j\rangle^{g_1}=\langle b_j\rangle$ and $\langle b_j\rangle^{g_1}=\langle a_j\rangle$ for $1\leqslant j\leqslant t$. Then $g_1$ interchanges $W$ and $W^g$ and hence interchanges $v$ and $v^g=w$. This implies that $\Gamma$ is undirected, a contradiction. Therefore, $X_v$ is not the stabilizer of a nontrivial proper subspace of $V$, and so $G\nleqslant\PGaL_n(q)$.

From now on assume that $\Gamma$ is $(G,3)$-arc-transitive and so in particular, $X$ acts $2$-arc-transitively on $\Gamma$. First assume that~(ii) holds. Since $X_v$ stabilizes $U$, the group $X_{vw}=X_v\cap X_v^g$ stabilizes $(U+U^g)/U$. As $\dim(U)<n/2$, we have
\[
\dim((U+U^g)/U)=\dim(U^g)-\dim(U\cap U^g)\leqslant\dim(U)<n-\dim(U)=\dim(V/U)
\]
and so $(U+U^g)/U<V/U$. Similarly, $X_{uv}$ stabilizes $(U+U^{g^{-1}})/U$ with $(U+U^{g^{-1}})/U<V/U$. Then since $X_v$ stabilizes $V/U$ and induces a group containing $\SL(V/U)$ and $X_v=X_{uv}X_{vw}$, we deduce from Lemma~\ref{lem21} that $(U+U^g)/U=0$ or $(U+U^{g^{-1}})/U=0$. Thus $U=U^g$. Since $X_v$ stabilizes $W$, we see that $X_{vw}$ stabilizes $W\cap W^g$ and $X_{uv}$ stabilizes $W\cap W^{g^{-1}}$. Moreover, $\dim(W)+\dim(W^g)=2n-2\dim(U)>n$, which implies that $W\cap W^g>0$ and $W\cap W^{g^{-1}}=(W\cap W^{g^{-1}})^g>0$. Then we deduce from Lemma~\ref{lem21} that $W\cap W^g=W$ or $W\cap W^{g^{-1}}=W$. This implies that $W=W^g$ and so $g\in X_v$, a contradiction.

Assume that~(iii) holds with $U=U^g$. In this case we have $W\neq W^g$ as $g\notin X_v$. Consequently, $W\cap W^g<W$ and $W\cap W^{g^{-1}}<W$. Moreover, $W\cap W^g\geqslant U$ and $W\cap W^{g^{-1}}\geqslant U$. Since $X_v$ stabilizes $U$ and $W$, the group $X_{vw}=X_v\cap X_v^g$ stabilizes $(W\cap W^g)/U$. Similarly, $X_{uv}$ stabilizes $(W\cap W^{g^{-1}})/U$. Then since $(W\cap W^g)/U<W/U$, $(W\cap W^{g^{-1}})/U<W/U$ and $X_v$ stabilizes $W/U$ and induces a group containing $\SL(W/U)$, we deduce from Lemma~\ref{lem21} that $(W\cap W^g)/U=0$ or $(W\cap W^{g^{-1}})/U=0$. This leads to $W\cap W^g=U$. Let $e_1,\dots,e_s$ be a basis of $U$. Then there exist $a_1,\dots,a_t$ and $b_1,\dots,b_t$ in $V$ such that $e_1,\dots,e_s,a_1,\dots,a_t$ is a basis of $W$ and $e_1,\dots,e_s,b_1,\dots,b_t$ is a basis of $W^g$. It follows that $e_1,\dots,e_s,a_1,\dots,a_t,b_1,\dots,b_t$ is a basis of $W+W^g$. Take $g_1$ to be an element of $X$ such that $\langle e_i\rangle^{g_1}=\langle e_i\rangle$ for $1\leqslant i\leqslant s$ and $\langle a_j\rangle^{g_1}=\langle b_j\rangle$ and $\langle b_j\rangle^{g_1}=\langle a_j\rangle$ for $1\leqslant j\leqslant t$. Then $g_1$ interchanges $W$ and $W^g$ and stabilizes $U=U^g$. This implies that $g_1$ interchanges $v$ and $v^g=w$, and so $\Gamma$ is undirected, a contradiction.

Assume that~(iii) holds with $W=W^g$. In this case we have $U\neq U^g$ as $g\notin X_v$. Consequently, $(U+U^g)/U>0$ and $(U+U^{g^{-1}})/U>0$. Moreover, $U\cap U^g<W$ and $U\cap U^{g^{-1}}<W$. Since $X_v$ stabilizes $U$, the group $X_{vw}=X_v\cap X_v^g$ stabilizes $(U+U^g)/U$. Similarly, $X_{uv}$ stabilizes $(U+U^{g^{-1}})/U$. Then since $X_v$ stabilizes $W/U$ with the action on $W/U$ containing $\SL(W/U)$, we deduce from Lemma~\ref{lem21} that $(U+U^g)/U=W/U$ or $(U+U^{g^{-1}})/U=W/U$. This leads to $U+U^g=W$.
Similarly as in the previous paragraph we may take an element $g_1$ of $X$ interchanging $v$ and $v^g=w$, which implies that $\Gamma$ is undirected, a contradiction.

Next assume that~(iii) holds with $U\neq U^g$ and $W\neq W^g$. Then $U\cap U^g<U$, $U\cap U^{g^{-1}}<U$, $(W+W^g)/W>0$ and $(W+W^{g^{-1}})/W>0$. Moreover, since $X_{vw}$ stabilizes $U\cap U^g$ and $X_{uv}$ stabilizes $U\cap U^{g^{-1}}$, we deduce from Lemma~\ref{lem21} that $U\cap U^g=0$. Since $X_{vw}$ stabilizes $(W+W^g)/W$, $X_{uv}$ stabilizes $(W+W^{g^{-1}})/W$ and the action of $X_v$ on $V/W$ contains $\SL(V/W)$, we deduce from Lemma~\ref{lem21} that $W+W^g=V$, so
\begin{equation}\label{eq3}
\dim(W\cap W^g)=2\dim(W)-n=n-2m=\dim(W)-\dim(U).
\end{equation}
Since $X_{vw}$ stabilizes $((W\cap U^g)+U)/U$, $X_{uv}$ stabilizes $((W\cap U^{g^{-1}})+U)/U$ and the action of $X_v$ on $W/U$ contains $\SL(W/U)$, we deduce from Lemma~\ref{lem21} that either $((W\cap U^g)+U)/U=0$ or $W/U$, or $((W\cap U^{g^{-1}})+U)/U=0$ or $W/U$. Without loss of generality, assume $((W\cap U^g)+U)/U=0$ or $W/U$. Then either $W\cap U^g\leqslant U$ or $(W\cap U^g)+U=W$.

In this paragraph we deal with the case when $W\cap U^g\leqslant U$. Since $U\cap U^g=0$, we have $(W^g\cap W)\cap U^g\leqslant W\cap U^g=(W\cap U^g)\cap U=0$ and so $(W\cap W^{g^{-1}})\cap U=0$. Hence~\eqref{eq3} implies that $W=(W\cap W^{g^{-1}})\oplus U$. Consider the action of $X_v$ on $W$. Since $X_{uv}$ stabilizes the complement $W\cap W^{g^{-1}}$ of $U$ in $W$ and $X_{vw}$ stabilizes $W\cap W^g$, we derive from Lemma~\ref{lem22} that either $W\cap W^g\leqslant U$ or $(W\cap W^g)+U=W$. If $(W\cap W^g)+U=W$, then~\eqref{eq3} implies that $W=(W\cap W^g)\oplus U$. However, we know that $W^g=(W\cap W^g)\oplus U^g$ as $W=(W\cap W^{g^{-1}})\oplus U$. Since $X_{W,W\cap W^g}$ induces $\SL(W/W\cap W^g)$ on $W/W\cap W^g$, this implies that there exists $g_1\in X$ stabilizing $W\cap W^g$ and interchanging $U$ and $U^g$ so that $g_1$ interchanges $v$ and $w=v^g$, contrary to $\Gamma$ being directed. Therefore $W\cap W^g\leqslant U$. As $W\cap U^g=0$ and $\dim(W)+\dim(U^g)=n$, we have $V=W\oplus U^g$. Now $X_v$ contains all lower unitriangular matrices and so $|X_v|_p\geqslant q^{n(n-1)/2}$, while $X_{vw}$ stabilizes the decomposition $V=W\oplus U^g$. Thus the valency of $\Gamma$ has $p$-part
\[
\frac{|X_v|_p}{|X_{vw}|_p}\geqslant q^{mn-m^2}.
\]
Since $\Gamma$ is $(G,3)$-arc-transitive, $(|X_v|/|X_{vw}|)^3$ divides $|G_v|$. Hence $(|X_v|/|X_{vw}|)^3$ divides $|\Out(L)||L_v|$, which yields
\[
q^{3mn-3m^2}\leqslant\frac{|X_v|_p^3}{|X_{vw}|_p^3}\leqslant2fq^{n(n-1)/2}.
\]
Thereby we deduce $q^{3mn-3m^2}\leqslant q\cdot q^{n(n-1)/2}<q^{n^2/2}$, which implies
\[
0<n^2-6mn+6m^2<n^2-5mn+6m^2=(n-2m)(n-3m).
\]
However, as $2m<n\leqslant 3m$, this is not possible.

Now we deal with the case when $(W\cap U^g)+U=W$. Then since $W\cap U^g\leqslant W\cap W^g$ and $\dim(W\cap W^g)+\dim(U)=\dim(W)$ by~\eqref{eq3}, we have $W=(W\cap W^g)\oplus U$ and $(W\cap W^g)\cap U=0$. It follows that
\[
W\cap U^{g^{-1}}=(W^g\cap U)^{g^{-1}}=((W\cap W^g)\cap U)^{g^{-1}}=0.
\]
Then the same argument as in the previous paragraph yields that this is not possible either.
\end{proof}

In what follows we deal with $\calC_2$-subgroups. Suppose that $X_v$ preserves a decomposition $V=W_1\oplus\dots\oplus W_k$ such that $\dim(W_1)=\dots=\dim(W_k)=m$ with $k\geqslant2$ and $n=mk$. For a subgroup $H$ of $G_v$ or $X_v$, denote by $\overline{H}$ the induced permutation group on $\{W_1,\dots,W_k\}$. Note from~\cite[Proposition~4.2.9]{KL1990} that $\overline{L_v}=\overline{X_v}=\Sy_k$.

\begin{lemma}\label{lem13}
Suppose that Hypothesis~$\ref{hyp1}$ holds and that $G_v$ is a $\calC_2$-subgroup. Then $n\geqslant3$.
\end{lemma}

\begin{proof}
Suppose that $\Gamma$ is $(G,2)$-arc-transitive with $n=2$. Then $m=1$, $k=2$, $L_v=\D_{2(q-1)/\gcd(p-1,2)}$ and $|\Out(L)|=(p-1,2)f$.

First assume that $f=1$. In this case, $p\geqslant5$ and $G_v=\D_{p-1}$ or $\D_{2(p-1)}$, where $G=\PSL_2(p)$ or $\PGL_2(p)$, respectively. Let $N$ be the unique cyclic subgroup of index $2$ of $G_v$. Then since $G_v=G_{uv}G_{vw}$, at least one of $G_{uv}$ or $G_{vw}$, say $G_{uv}$, is not contained in $N$. This implies that $G_{vw}$ is not contained in $N$ since $G_{vw}\cong G_{uv}$. Consequently, $G_{uv}\cap N$ and $G_{vw}\cap N$ are the unique cyclic subgroups of index $2$ of $G_{uv}$ and $G_{vw}$, respectively. Thus we conclude that $G_{uv}\cap N$ and $G_{vw}\cap N$ are subgroups of the cyclic group $N$ of the same order, and so $G_{uv}\cap N=G_{vw}\cap N$. Moreover, as $G_{vw}\cap N^{\overline{g}}=(G_{uv}\cap N)^{\overline{g}}\cong G_{uv}\cap N$ is a cyclic subgroup of index $2$ of $G_{vw}$, we deduce that $G_{vw}\cap N^{\overline{g}}=G_{vw}\cap N$ and hence $(G_{uv}\cap N)^{\overline{g}}=G_{vw}\cap N^{\overline{g}}=G_{vw}\cap N=G_{uv}\cap N$. Since $G_{uv}\cap N$ is characteristic in $N$ and hence normal in $G_v$, this contradicts Lemma~\ref{lem26}.

Next assume that $f\geqslant2$ with $\ppd(p,f)\neq\emptyset$. In this case, take any $r\in\ppd(p,f)$. We have $r>f$ and so $r$ is coprime to $|\Out(L)|$. Consequently, there is a unique subgroup $M$ of order $r$ in $G_v$. Since $r\in\pi(G_v)=\pi(G_{uv})=\pi(G_{vw})$, it follows that $M\leqslant G_{uv}$ and $M\leqslant G_{vw}$. Moreover, since $G_{uv}^{\overline{g}}=G_{vw}$ we have $M^{\overline{g}}=M$, again contradicting Lemma~\ref{lem26}.

Next assume that $f\geqslant2$ with $\ppd(p,f)=\emptyset$. Then $f=2$ and $p$ is a Mersenne prime. If $q=9$, then computation in \magma~\cite{magma} shows that $G_v$ does not have a factorization $G_v=G_{uv}G_{vw}$ with $|G_v|/|G_{uv}|\geqslant2$ such that $G_{uv}$ and $G_{vw}$ are conjugate in $G$, a contradiction. Therefore, $r\geqslant7$ and so $p-1$ has an odd prime divisor $r$. Then along the same lines as the previous case we see that this is not possible.
\end{proof}

For the rest of this section we assume that Hypothesis~\ref{hyp1} holds and $G_v$ is a $\calC_2$-subgroup of $G$. Under this assumption we have $n\geqslant3$ by Lemma~\ref{lem13}. We make the additional assumption that $\Gamma$ is $(G,3)$-arc-transitive. Recall the notation introduced at the beginning of Section~\ref{sec5}.

\begin{lemma}\label{lem11}
If $m=1$, then at least one of $\overline{X_{uv}}$ or $\overline{X_{vw}}$ is intransitive.
\end{lemma}

\begin{proof}
Suppose that $m=1$ while both $\overline{X_{uv}}$ and $\overline{X_{vw}}$ are transitive. Let $M$ be the subgroup of $X$ stabilizing each of $W_1,\dots,W_n$. Then $M=\C_{q-1}^{n-1}$, $\overline{X_v}=X_v/M$, $\overline{X_{uv}}=X_{uv}M/M$ and $\overline{X_{vw}}=X_{vw}M/M$. From the factorization $X_v=X_{uv}X_{vw}$ we deduce that $\overline{X_v}=\overline{X_{uv}}\,\overline{X_{vw}}$. Then since $M$ is abelian, Lemma~\ref{lem9} implies that $\overline{X_{uv}}=\overline{X_{vw}}=\Sy_n$. In particular, $X_{uv}\cap M=\Rad(X_{uv})$ and $X_{vw}\cap M=\Rad(X_{vw})$, and so as $X_{uv}\cong X_{vw}$, we have that $X_{uv}\cap M\cong X_{vw}\cap M$.
As $G_v$ is maximal in $G$, we have $q\geqslant5$ (see~\cite{BHR2013} and~\cite[Table~3.5.H]{KL1990}).

Let $r$ be a prime divisor of $|X_{uv}\cap M|=|X_{vw}\cap M|$. Since $\overline{X_{uv}}=\Sy_n$, we have $X_v=MX_{uv}>M\rtimes\A_n$ such that $\Omega_r(M)$ is the deleted permutation module of $\A_n$ over $\bbF_r$. As $\Omega_r(X_{uv}\cap M)$ is characteristic in $X_{uv}\cap M$, and $X_{uv}\cap M$ is normal in $MX_{uv}$, the elementary abelian $r$-group $\Omega_r(X_{uv}\cap M)$ is normal in $MX_{uv}$ and so is a submodule of the deleted permutation module of $\A_n$. Similarly, $\Omega_r(X_{vw}\cap M)$ is also a permutation submodule of $\A_n$. From~\cite{Mortimer1980} we know that the submodules of the permutation module $\Omega_r(M)$ of $\A_n$ are $0$, $\Omega_r(M)$, and a unique submodule of dimension $1$ if $r$ divides $n$. Therefore, $\Omega_r(X_{uv}\cap M)=\Omega_r(X_{vw}\cap M)$. It follows that $\Omega_r(X_{uv}\cap M)$ is normalized by $g$ as
\begin{align*}
\Omega_r(X_{uv}\cap M)^g&=\Omega_r(\Rad(X_{uv}))^g\\
&=\Omega_r(\Rad(X_{vw}))=\Omega_r(X_{vw}\cap M)=\Omega_r(X_{uv}\cap M).
\end{align*}
Clearly, $\Omega_r(X_{uv}\cap M)=\Omega_r(\Rad(X_{uv}))$ is normal in both $M$ and $X_{uv}$. Moreover, $X_v=MX_{uv}$ and so we conclude that $\Omega_r(X_{uv}\cap M)$ is normal in $\langle M,X_{uv},g\rangle=\langle X_v,g\rangle=X$. This yields $\Omega_r(X_{uv}\cap M)\leqslant\Z(X)$. In particular, $\Omega_r(X_{uv}\cap M)$ is cyclic, which implies that $|X_{uv}\cap M|_r$ divides $q-1$ since $X_{uv}\cap M\leqslant M=\C_{q-1}^{n-1}$.

Now we know that $|X_{uv}\cap M|_r$ divides $q-1$ for each prime divisor $r$ of $|X_{uv}\cap M|$. As a consequence, $|X_{uv}\cap M|$ divides $q-1$, and so $|X_{uv}|=|X_{uv}\cap M||\overline{X_{uv}}|$ divides $(q-1)n!$. Then as $|X_v|$ divides $|X_{uv}|^2$, it follows that $(q-1)^{n-1}n!=|X_v|$ divides $(q-1)^2(n!)^2$. Hence
\begin{equation}\label{eq6}
(q-1)^{n-3}\ \vert\ n!.
\end{equation}
Furthermore, since $\Gamma$ is $(G,3)$-arc-transitive, $(|X_v|/|X_{uv}|)^3$ divides $|G_v|$ and hence divides $|\Out(L)||L_v|=2f(q-1)^{n-1}n!$. Consequently, $(q-1)^{3n-3}(n!)^3=|X_v|^3$ divides $2f(q-1)^{n-1}n!|X_{uv}|^3$, which implies that $(q-1)^{2n-2}(n!)^2$ divides $2f|X_{uv}|^3$. This together with the conclusion that $|X_{uv}|$ divides $(q-1)n!$ leads to
\begin{equation}\label{eq7}
(q-1)^{2n-5}\ \vert\ 2fn!
\end{equation}

First assume $n=3$. Then~\eqref{eq7} turns out to be
\begin{equation}\label{eq8}
(q-1)\ \vert\ 12f.
\end{equation}
We deduce that $p-1\leqslant(p^f-1)/f\leqslant12$ and $2^f-1\leqslant p^f-1\leqslant12f$, which lead to $p\leqslant13$ and $f\leqslant6$, respectively. Checking~\eqref{eq8} for $q=p^f\geqslant3$ with $p\leqslant13$ and $f\leqslant6$ we obtain $q\in\{5,7,9,13,25\}$. However, for these values of $q$, computation in \magma~\cite{magma} shows that there is no nontrivial homogeneous factorization of $X_v$ with the two factors conjugate in $X$, a contradiction.

Next assume $n=4$. Then~\eqref{eq7} turns out to be
\begin{equation}\label{eq9}
(q-1)^3\ \vert\ 48f.
\end{equation}
We deduce that $(p-1)^3\leqslant(p^f-1)^3/f\leqslant48$ and $(2^f-1)^3\leqslant(p^f-1)^3\leqslant48f$, which lead to $p\leqslant3$ and $f\leqslant2$, respectively. However, there is no such pair $(p,f)$ such that $q=p^f\geqslant5$ satisfies~\eqref{eq9}, a contradiction.

Finally assume that $n\geqslant5$. Suppose that $q-1$ is divisible by an odd prime, say $r$. Then we derive from~\eqref{eq6} that $r^{n-3}\leqslant(n!)_r<r^{n/(r-1)}\leqslant r^{n/2}$, which forces $n=5$. However, then~\eqref{eq6} implies that $r^2$ divides $5!=2^3\cdot3\cdot5$, which is not possible. Consequently, $q-1$ is a power of $2$. Then~\eqref{eq6} yields $(q-1)^{n-3}\leqslant(n!)_2<2^n$ and so $q-1<2^{n/(n-3)}\leqslant2^{5/2}$, which implies $q=5$. However, substituting $q=5$ into~\eqref{eq7} we obtain $4^{2n-5}\leqslant2(n!)_2<2^{n+1}$, contradicting $n\geqslant5$.
\end{proof}

\begin{lemma}\label{lem12}
If $m=1$, then $n\leqslant6$.
\end{lemma}

\begin{proof}
Suppose that $m=1$ and $n\geqslant7$. Let $M$ be the subgroup of $X_v$ stabilizing each of $W_1,\dots,W_n$. Then since $M$ is abelian, $\overline{X_{uv}}$ has the same set of insoluble composition factors as $X_{uv}$ and $\overline{X_{vw}}$ has the same set of insoluble composition factors as $X_{vw}$. Since $X_{uv}\cong X_{vw}$, it follows that $\overline{X_{uv}}$ and $\overline{X_{vw}}$ have the same set of insoluble composition factors. From the factorization $X_v=X_{uv}X_{vw}$ we deduce that $\overline{X_v}=\overline{X_{uv}}\,\overline{X_{vw}}$ and hence by Lemma~\ref{lem10} both $\overline{X_{uv}}$ and $\overline{X_{vw}}$ contain $\A_n$. However, this implies that both $\overline{X_{uv}}$ and $\overline{X_{vw}}$ are transitive, contrary to Lemma~\ref{lem11}. This completes the proof.
\end{proof}

\begin{lemma}\label{lem8}
$m\geqslant2$.
\end{lemma}

\begin{proof}
Suppose that $m=1$. Then we have by Lemmas~\ref{lem13} and~\ref{lem12} that $3\leqslant n\leqslant6$. As $G_v$ is maximal in $G$, we have $q\geqslant5$ (see~\cite{BHR2013} and~\cite[Table~3.5.H]{KL1990}). Let $M$ be the subgroup of $X$ stabilizing each of $W_1,\dots,W_n$. Then $M\cong\C_{q-1}^{n-1}$, $\overline{X_v}=X_v/M$, $\overline{X_{uv}}=X_{uv}M/M$ and $\overline{X_{vw}}=X_{vw}M/M$. From the factorization $X_v=X_{uv}X_{vw}$ we deduce that $\Sy_n=\overline{X_v}=\overline{X_{uv}}\,\overline{X_{vw}}$. Since $X_{uv}\cap M$ and $X_{vw}\cap M$ are normal abelian subgroups of $X_{uv}$ and $X_{vw}$, respectively, we have $X_{uv}\cap M\leqslant F(X_{uv})$ and $X_{vw}\cap M\leqslant F(X_{vw})$. Then as $X_{uv}\cong X_{vw}$, it follows that
\begin{align}\label{eq5}
\overline{X_{uv}}/(F(X_{uv})/(X_{uv}\cap M))&\cong X_{uv}/F(X_{uv})\\\nonumber
&\cong X_{vw}/F(X_{vw})\cong\overline{X_{vw}}/(F(X_{vw})/(X_{vw}\cap M)).
\end{align}
Note that $F(X_{uv})/(X_{uv}\cap M)$ and $F(X_{vw})/(X_{vw}\cap M)$ are both nilpotent. We conclude that the factors $\overline{X_{uv}}$ and $\overline{X_{vw}}$ of the factorization $\Sy_n=\overline{X_{uv}}\,\overline{X_{vw}}$ have isomorphic factor groups by nilpotent subgroups. This shows that, interchanging $\overline{X_{uv}}$ and $\overline{X_{vw}}$ if necessary, either the pair $(\overline{X_{uv}},\overline{X_{vw}})$ lies in Table~\ref{tab3} below, or both $\overline{X_{uv}}$ and $\overline{X_{vw}}$ are transitive. The latter is not possible by Lemma~\ref{lem11}. Thus to finish the proof, we only need to exclude the candidates in Table~\ref{tab3}.

\begin{table}[htbp]
\caption{The pair $(\overline{X_{uv}},\overline{X_{vw}})$ in the proof of Lemma~\ref{lem8}}\label{tab3}
\centering
\begin{tabular}{|l|l|l|l|}
\hline
row & $n$ & $\overline{X_{uv}}$ & $\overline{X_{vw}}$\\
\hline
1 & $3$ & $\C_2$ & $\C_3$\\
2 & $3$ & $\C_2$ & $\Sy_3$\\
3 & $4$ & $\C_2^2$ & $\Sy_3$\\
4 & $4$ & $\C_4$ & $\Sy_3$\\
5 & $4$ & $\D_8$ & $\C_3$\\
6 & $4$ & $\D_8$ & $\Sy_3$\\
7 & $4$ & $\Sy_4$ & $\Sy_3$\\
8 & $6$ & $\PGL_2(5)$ & $\Sy_5$\\
\hline
\end{tabular}
\end{table}

\underline{Rows~1--6.} For these rows, $\bfO_3(X_{uv}\cap M)$ is a Sylow $3$-subgroup of $X_{uv}$ and $|\overline{X_{vw}}|_3=3$. Note that $\bfO_3(X_{uv}\cap M)$ is a normal abelian subgroup of $X_{uv}$. Then from $X_{uv}\cong X_{vw}$ we conclude that $X_{vw}$ has a unique Sylow $3$-subgroup $P$ and $P$ is a normal abelian subgroup of $X_{vw}$. Thus $MP$ is normal in $MX_{vw}$ and contains every $3$-element of $MX_{vw}$. Since $P$ and $X_{vw}\cap M$ are both normal abelian subgroups of $X_{vw}$, we have $P\leqslant F(X_{vw})$ and $X_{vw}\cap M\leqslant F(X_{vw})$. Thus $\bfO_r(X_{vw}\cap M)$ is centralized by $P$ for every prime $r\neq3$. As $\bfO_r(X_{vw}\cap M)$ is centralized by $M$, it follows that $\bfO_r(X_{vw}\cap M)$ is centralized by $MP$ and hence by every $3$-element of $MX_{vw}$. Since $|MX_{vw}|_3=|X_v|_3$, for any $3$-element $y$ of $X_v$ there exists $z\in X_v$ such that $zyz^{-1}\in MX_{vw}$ and so $zyz^{-1}\in\Cen_{X_v}(\bfO_r(X_{vw}\cap M))$, which is equivalent to $y\in\Cen_{X_v}(\bfO_r((X_{vw}\cap M)^z))$.

First assume $n=3$. Write each element $x$ of $M$ as $x=(x_1,x_2,x_3)$, where $x_i\in\GL_1(q)$ for $1\leqslant i\leqslant3$ such that $x_1x_2x_3=1$. Let
\[
y=
\begin{pmatrix}
0&1&0\\
0&0&1\\
1&0&0
\end{pmatrix}.
\]
Then $y\in X_v$ and $y$ has order $3$. Thus, for any prime $r\neq3$ there exists $z\in X_v$ such that $\bfO_r((X_{vw}\cap M)^z)$ is centralized by $y$. It follows that for any $x=(x_1,x_2,x_3)\in\bfO_r((X_{vw}\cap M)^z)\leqslant M$, the conclusion $x^y=x$ gives $(x_2,x_3,x_1)=(x_1,x_2,x_3)$ and then the condition $x_1x_2x_3=1$ implies that $x$ has order dividing $3$. Hence for any prime $r\neq3$, $\bfO_r((X_{vw}\cap M)^z)=1$, that is, $|X_{vw}\cap M|_r=1$. Accordingly, $|X_{vw}\cap M|$ is a power of $3$. Since $|X_v|$ divides $|X_{vw}|^2$ and $|X_{vw}|=|\overline{X_{vw}}||X_{vw}\cap M|$ divides $6|X_{vw}\cap M|$, we deduce that $|X_v|$ divides $(6|X_{vw}\cap M|)^2$. As $|X_v|=6(q-1)^2$, it then follows that $q-1$ is a power of $3$, which contradicts Mih\v{a}ilescu's theorem~\cite{Mihailescu2004} as $q\geqslant5$.

Next assume $n=4$. Write each element $x$ of $M$ as $x=(x_1,x_2,x_3,x_4)$, where $x_i\in\GL_1(q)$ for $1\leqslant i\leqslant3$ such that $x_1x_2x_3x_4=1$. Let
\[
y=
\begin{pmatrix}
0&1&0&0\\
0&0&1&0\\
1&0&0&0\\
0&0&0&1
\end{pmatrix}.
\]
Then $y\in X_v$ and $y$ has order $3$, so there exists $z\in X_v$ such that $\bfO_2((X_{vw}\cap M)^z)$ is centralized by $y$. Thus, for any $x=(x_1,x_2,x_3,x_4)\in\bfO_2((X_{vw}\cap M)^z)\leqslant M$, we deduce from the conclusion $x^y=x$ that $(x_2,x_3,x_1,x_4)=(x_1,x_2,x_3,x_4)$, and so the condition $x_1x_2x_3x_4=1$ implies that $x=(x_1,x_1,x_1,x_1^{-3})$. Consequently, $|\bfO_2(X_{vw}\cap M)|=|\bfO_2((X_{vw}\cap M)^z)|$ divides $q-1$ and hence $(q-1)_2$. Since $|\overline{X_{vw}}|_2$ divides $2$, we then conclude that $|X_{vw}|_2$ divides $2(q-1)_2$. This is contrary to the condition that $|X_{vw}|^2$ is divisible by $|X_v|=24(q-1)^3$.

\underline{Row~7.} Since the normal nilpotent subgroups of $\overline{X_{uv}}$ are $1$ and $\C_2^2$, we derive from~\eqref{eq5} that $F(X_{uv})/(X_{uv}\cap M)=\C_2^2$ and $F(X_{vw})/(X_{vw}\cap M)=1$. In particular, $F(X_{vw})=X_{vw}\cap M$ is abelian, which implies that $F(X_{uv})$ is abelian. Consequently, $X_{uv}\cap M$ is centralized by $F(X_{uv})$ and hence by $MF(X_{uv})$. For any element $x=(x_1,x_2,x_3,x_4)$ in $X_{uv}\cap M$, where $x_i\in\GL_1(q)$ for $1\leqslant i\leqslant4$ such that $x_1x_2x_3x_4=1$, we then have $x_1=x_2=x_3=x_4$ since $MF(X_{uv})/M\cong F(X_{uv})/(X_{uv}\cap M)=\C_2^2$ is a transitive subgroup of $\Sy_4$. This further implies that $x_1^4=1$, whence $|X_{uv}\cap M|$ divides $4$. Now $|X_{uv}|=|\overline{X_{uv}}||X_{uv}\cap M|=24|X_{uv}\cap M|$ divides $24\cdot4$ while $24(q-1)^3=|X_v|$ divides $|X_{uv}|^2$. We deduce that $(q-1)^3$ divides $24\cdot4^2$, which leads to $q=5$. It follows that $\Gamma$ has valency $|X_v|/|X_{uv}|$ which is divisible by $(q-1)^3/4=4^2$, and so $|G_v|$ is divisible by $4^6$ as $\Gamma$ is $(G,3)$-arc-transitive. However, $|G_v|=|L_v||\Out(L)|$ divides $24(q-1)^3\cdot2=3\cdot4^5$, a contradiction.

\underline{Row~8.} Since $X_{uv}/(X_{uv}\cap M)$ and $X_{vw}/(X_{vw}\cap M)$ are almost simple groups, we have $X_{uv}\cap M=\Rad(X_{uv})$ and $X_{vw}\cap M=\Rad(X_{vw})$. Consequently, $X_{uv}\cap M\cong X_{vw}\cap M$ as $X_{uv}\cong X_{vw}$. Let $r$ be a prime divisor of $|X_{uv}\cap M|=|X_{vw}\cap M|$. Then $\Omega_r(X_{uv}\cap M)\cong\Omega_r(X_{vw}\cap M)>1$. Note that $\Omega_r(M)$ is a permutation module of both $\overline{X_{uv}}=\PGL_2(5)$ and $\overline{X_{vw}}=\Sy_5$ over $\bbF_r$. As $\Omega_r(X_{uv}\cap M)$ is characteristic in $X_{uv}\cap M$ and $X_{uv}\cap M$ is normal in $MX_{uv}$, the elementary abelian $r$-group $\Omega_r(X_{uv}\cap M)$ is normal in $MX_{uv}$ and so is a permutation submodule of $\PGL_2(5)$. For the same reason, $\Omega_r(X_{vw}\cap M)$ is a permutation submodule of $\Sy_5$. From~\cite{Mortimer1980} we know that all the submodules of the permutation module $\Omega_r(M)$ of $\PGL_2(5)$ are $0$, $\Omega_r(M)$, a unique submodule of dimension $1$ and a unique submodule of dimension $5$. Therefore, $|\Omega_r(X_{uv}\cap M)|=|\Omega_r(X_{vw}\cap M)|=r$, $r^5$ or $r^6$. If $|\Omega_r(X_{vw}\cap M)|\geqslant r^5$, then by~\cite{Mortimer1980} the permutation module $\Omega_r(X_{vw}\cap M)$ of $\Sy_5$ has a submodule of dimension $4$, which implies that $\Omega_r(\Rad(X_{vw}))=\Omega_r(X_{vw}\cap M)$ contains a normal subgroup of $X_{vw}$ of order $r^4$. This would further imply that $\Omega_r(X_{uv}\cap M)$ contains a normal subgroup of $X_{uv}$ of order $r^4$, and so $\Omega_r(X_{uv}\cap M)$ contains a submodule of $\overline{X_{uv}}=\PGL_2(5)$ of dimension $4$. However, the permutation module $\Omega_r(M)$ of $\PGL_2(5)$ has no submodule of dimension $4$, a contradiction. Thus $|\Omega_r(X_{uv}\cap M)|=|\Omega_r(X_{vw}\cap M)|=r$.

Now we have $|\Omega_r(X_{uv}\cap M)|=r$ for each prime divisor $r$ of $|X_{uv}\cap M|$. This implies that $\bfO_r(X_{uv}\cap M)$ is cyclic and hence has order dividing $q-1$ for each prime divisor $r$ of $|X_{uv}\cap M|$. Consequently, $|X_{uv}\cap M|$ divides $q-1$. It follows that $|X_{uv}|=|\overline{X_{uv}}||X_{uv}\cap M|=120|X_{uv}\cap M|$ divides $120(q-1)$ while $|\Sy_6|(q-1)^5=|X_v|$ divides $|X_{uv}|^2$. We deduce that $|\Sy_6|(q-1)^3$ divides $120^2$, which contradicts $q\geqslant5$.
\end{proof}

\begin{lemma}\label{lem14}
If $m=2$, then $q\geqslant4$.
\end{lemma}

\begin{proof}
Suppose $m=2$ and $q=2$ or $3$. Since $G_v$ is maximal in $G$, \cite[Proposition~2.3.6]{BHR2013} implies that $q=3$. Let $M$ be the subgroup of $X_v$ stabilizing each of $W_1,\dots,W_k$ and let $\varphi_i$ be the action of $M$ on $W_i$ for $1\leqslant i\leqslant k$. Then $M$ is normal in $X_v$, $X_v/M=\overline{X_v}=\Sy_k$, $X_{uv}M/M=\overline{X_{uv}}$ and $X_{vw}M/M=\overline{X_{vw}}$. Note that $\SL_2(3)\wr\Sy_k\leqslant X_v\leqslant\GL_2(3)\wr\Sy_k$. In fact, $X_v=(\SL_2(3)^k\rtimes\C_2^{k-1}).\Sy_k$. For $2\leqslant k\leqslant5$, computation in \magma~\cite{magma} shows that there is no nontrivial homogeneous factorization of $X_v$ with the two factors conjugate in $X$. Therefore, $k\geqslant6$.

From the factorization $X_v=X_{uv}X_{vw}$ we deduce that $\Sy_k=\overline{X_v}=\overline{X_{uv}}\,\overline{X_{vw}}$. Since $X_{uv}\cong X_{vw}$ and $M$ is soluble, we conclude that $\overline{X_{uv}}$ and $\overline{X_{vw}}$ have the same set of insoluble composition factors. If $k\geqslant7$, then by Lemma~\ref{lem10}, both $\overline{X_{uv}}$ and $\overline{X_{vw}}$ contain $\A_k$. If $k=6$, then since $X_{uv}\cong X_{vw}$ and $|M|_5=1$, the two factors of the factorization $\Sy_6=\overline{X_{uv}}\,\overline{X_{vw}}$ both have order divisible by $5$. This together with the condition that $\overline{X_{uv}}$ and $\overline{X_{vw}}$ have the same set of insoluble composition factors implies that $\overline{X_{uv}}$ and $\overline{X_{vw}}$ are both almost simple groups with socle $\A_5$ or $\A_6$. To sum up, we have two cases:
\begin{itemize}
\item[(i)] both $\overline{X_{uv}}$ and $\overline{X_{vw}}$ contain $\A_k$;
\item[(ii)] $k=6$ and both $\overline{X_{uv}}$ and $\overline{X_{vw}}$ are almost simple groups with socle $\A_5$.
\end{itemize}
In particular, $\overline{X_{uv}}$ and $\overline{X_{vw}}$ are always almost simple groups. Accordingly, $X_{uv}\cap M=\Rad(X_{uv})\cong\Rad(X_{vw})=X_{vw}\cap M$ and so $\overline{X_{uv}}\cong\overline{X_{vw}}$.

First assume that~(i) occurs. Then $\overline{X_{uv}}=\overline{X_{vw}}=\Sy_k$. Note that $\Omega_2(M')=\Z(M')=\Z(M)=\C_2^k$. We have $\Omega_2((X_{uv}\cap M)')\leqslant\Omega_2(M')=\Z(M)$. As a consequence, $\Omega_2((X_{uv}\cap M)')$ is normal in $M$. Also, $\Omega_2((X_{uv}\cap M)')$ is normal in $X_{uv}$ since it is characteristic in $X_{uv}\cap M$ and $X_{uv}\cap M$ is normal in $X_{uv}$. Hence $\Omega_2((X_{uv}\cap M)')$ is normal in $MX_{uv}$. Since $\overline{X_{uv}}=\Sy_k$, we see that $\Omega_2((X_{uv}\cap M)')$ is a submodule of the permutation module $\Z(M)$ of $\Sy_k$ over $\bbF_2$. Similarly, $\Omega_2((X_{vw}\cap M)')$ is a submodule of the same permutation module $\Z(M)$ of $\Sy_k$. Since $X_{uv}\cong X_{vw}$, we derive that
\[
\Omega_2((X_{uv}\cap M)')=\Omega_2(\Rad(X_{uv})')\cong\Omega_2(\Rad(X_{vw})')=\Omega_2((X_{vw}\cap M)').
\]
From~\cite{Mortimer1980} we know that all the submodules of the permutation module $\Z(M)$ of $\Sy_k$ are $0$, $\Z(M)$, a submodule of dimension $1$ and a submodule of dimension $k-1$. Therefore, $\Omega_2((X_{uv}\cap M)')=\Omega_2((X_{vw}\cap M)')$ and is normal in both $X_{uv}$ and $X_{vw}$ and hence in $\langle X_{uv},X_{vw}\rangle=X_v$. Moreover,
\begin{align*}
\Omega_2((X_{uv}\cap M)')^g&=\Omega_2(\Rad(X_{uv})')^g\\
&=\Omega_2(\Rad(X_{vw})')=\Omega_2((X_{vw}\cap M)')=\Omega_2((X_{uv}\cap M)').
\end{align*}
Thus $\Omega_2((X_{uv}\cap M)')$ is normal in $\langle X_v,g\rangle=X$, and so $\Omega_2((X_{uv}\cap M)')\leqslant\Z(X)$. In particular, any nontrivial $z\in\Omega_2((X_{uv}\cap M)')$ will satisfy $\varphi_i(z)\neq1$ for all $1\leqslant i\leqslant k$. Suppose $|(X_{uv}\cap M)'|_2>8$. Then since $\varphi_1((X_{uv}\cap M)')\leqslant\varphi_1(M')=\SL_2(3)$ and $|\SL_2(3)|_2=8$, there exist $2$-elements $x$ and $y$ of $(X_{uv}\cap M)'$ such that $x\neq y$ and $\varphi_1(x)=\varphi_1(y)$. Note that every $2$-element of $M'=\SL_2(3)^k$ has order dividing $4$. If $xy^{-1}$ has order $2$, then $xy^{-1}$ is a nontrivial element in $\Omega_2((X_{uv}\cap M)')$ with $\varphi_1(xy^{-1})=\varphi_1(x)\varphi_1(y)^{-1}=1$, a contradiction. If $xy^{-1}$ has order $4$, then $(xy^{-1})^2$ is a nontrivial element in $\Omega_2((X_{uv}\cap M)')$ with $\varphi_1((xy^{-1})^2)=(\varphi_1(x)\varphi_1(y)^{-1})^2=1$, still a contradiction. Thus $|(X_{uv}\cap M)'|_2\leqslant8$.

Since $\overline{X_{uv}}$ is transitive, Lemma~\ref{lem3} implies that
\[
\varphi_1(X_{uv}\cap M)\cong\cdots\cong\varphi_k(X_{uv}\cap M)
\]
and $\pi(\varphi_1(X_{uv}\cap M))\supseteq\pi(\SL_2(3))=\{2,3\}$. Recall that $\overline{X_{uv}}=\overline{X_{vw}}=\Sy_k$. If $|\varphi_1(X_{uv}\cap M)|_2\leqslant4$, then the valency of $\Gamma$ has $2$-part
\[
\frac{|X_v|_2}{|X_{uv}|_2}=\frac{|M|_2}{|X_{uv}\cap M|_2}\geqslant\frac{|M|_2}{|\varphi_1(X_{uv}\cap M)|_2^k}\geqslant\frac{|M|_2}{4^k}=2^{2k-1}
\]
and so since $\Gamma$ is $(G,3)$-arc-transitive, $|G_v|_2\geqslant2^{3(2k-1)}$. However,
\[
|G_v|_2\leqslant|\Out(L)|_2|L_v|_2=2^{4k}(k!)_2<2^{5k},
\]
a contradiction. Thus $|\varphi_1(X_{uv}\cap M)|_2\geqslant8$, which in conjunction with the conclusion $\pi(\varphi_1(X_{uv}\cap M))\supseteq\{2,3\}$ indicates that $|\varphi_1(X_{uv}\cap M)|$ is divisible by $24$.

Let $\varphi_{1,2,3}$ be the action of $M$ on $W_1\oplus W_2\oplus W_3$, and $Y=\varphi_{1,2,3}(X_{uv}\cap M)$. Then
\[
|Y'|_2=|\varphi_{1,2,3}((X_{uv}\cap M)')|_2\leqslant|(X_{uv}\cap M)'|_2\leqslant8.
\]
Clearly, $\varphi_i(Y)=\varphi_i(X_{uv}\cap M)$ for $i=1,2,3$. Thus $|\varphi_1(Y)|=|\varphi_2(Y)|=|\varphi_3(Y)|$ is divisible by $24$. Since $\Gamma$ is $(G,3)$-arc-transitive and the $3$-part of the valency of $\Gamma$ is
\[
\frac{|X_v|_3}{|X_{uv}|_3}\geqslant\frac{|M|_3}{|X_{uv}\cap M|_3}=\frac{3^k}{|X_{uv}\cap M|_3},
\]
we deduce that
\[
\left(\frac{3^k}{|X_{uv}\cap M|_3}\right)^3\leqslant|G_v|_3\leqslant|\Out(L)|_3|L_v|_3=3^k(k!)_3<3^{3k/2}
\]
and hence $|X_{uv}\cap M|_3>3^{k/2}$. On the other side, $|X_{uv}\cap M|_3\leqslant|Y|_3^{\lceil k/3\rceil}$. It follows that $|Y|_3^{\lceil k/3\rceil}>3^{k/2}$, which implies $|Y|_3>3$. However, by a \magma~\cite{magma} computation there is no subgroup $Y$ of $\varphi_{1,2,3}(M)\leqslant\GL_2(3)\times\GL_2(3)\times\GL_2(3)$ with $|Y'|_2\leqslant8$ and $|Y|_3>3$ such that $|\varphi_1(Y)|=|\varphi_2(Y)|=|\varphi_3(Y)|$ is divisible by $24$, a contradiction.

Next assume that~(ii) occurs. Since $\overline{X_{uv}}\cong\overline{X_{vw}}$, we deduce from the factorization $\Sy_6=\overline{X_{uv}}\,\overline{X_{vw}}$ that $\overline{X_{uv}}\cong\overline{X_{vw}}\cong\Sy_5$. Since $\Gamma$ is $(G,3)$-arc-transitive, $(|X_v|/|X_{uv}|)^3$ divides $|G_v|$ and hence divides $|\Out(L)||L_v|=2^{28}\cdot3^8\cdot5$. Consequently, $2^{81}\cdot3^{24}\cdot5^3=|X_v|^3$ divides $2^{28}\cdot3^8\cdot5|X_{uv}|^3$, which implies that $2^{15}\cdot3^6\cdot5$ divides $|X_{uv}|$. However, computation in \magma~\cite{magma} shows that for $X=\SL_{12}(3)$ there is no homogeneous factorization $X_v=X_{uv}X_{vw}$ with $|X_{uv}|$ divisible by $2^{15}\cdot3^6\cdot5$ and $X_{uv}/\Rad(X_{uv})\cong X_{vw}/\Rad(X_{vw})\cong\Sy_5$, a contradiction.
\end{proof}

We are now able to rule out $\calC_2$-subgroups.

\begin{lemma}\label{lem18}
If $G_v$ is a $\calC_2$-subgroup of $G$, then $\Gamma$ is not $(G,3)$-arc-transitive.
\end{lemma}

\begin{proof}
Suppose that $G_v$ is a $\calC_2$-subgroup of $G$ while $\Gamma$ is $(G,3)$-arc-transitive. Then Lemma~\ref{lem8} shows that $m\geqslant2$. Moreover, if $m=2$ then $q\geqslant4$ by Lemma~\ref{lem14}. Let $M$ be the subgroup of $X_v$ stabilizing each of $W_1,\dots,W_k$ and let $\varphi_i$ be the action of $M$ on $W_i$ for $1\leqslant i\leqslant k$. Then $M$ is normal in $X_v$, $X_v/M=\overline{X_v}=\Sy_k$, $X_{uv}M/M=\overline{X_{uv}}$ and $X_{vw}M/M=\overline{X_{vw}}$. From the factorization $X_v=X_{uv}X_{vw}$ we deduce that $\overline{X_v}=\overline{X_{uv}}\,\overline{X_{vw}}$. Then by Lemma~\ref{lem1}, at least one of $\overline{X_{uv}}$ or $\overline{X_{vw}}$, say $\overline{X_{uv}}$, is a transitive subgroup of $\Sy_k$. Note that
\begin{equation}\label{eq4}
X_v=(\SL_m(q)^k\rtimes\C_{q-1}^{k-1})\rtimes\Sy_k,
\end{equation}
and so it follows from Lemma~\ref{lem3} that
\[
\varphi_1(X_{uv}\cap M)\cong\cdots\cong\varphi_k(X_{uv}\cap M)
\]
and $\pi(\SL_m(q))\subseteq\pi(\varphi_1(X_{uv}\cap M))$. Let $Z$ be the center of $\varphi_1(M)=\GL_m(q)$. Then $Z=\C_{q-1}$ and $\varphi_1(M)/Z=\PGL_m(q)$, which implies that
\[
\pi(\varphi_1(X_{uv}\cap M)Z/Z)\supseteq\pi(\varphi_1(X_{uv}\cap M))\setminus\pi(Z)\supseteq\pi(\PSL_m(q))\setminus\pi(q-1).
\]
Thereby we deduce from~\cite[Theorem~4]{LPS2000} that either $\varphi_1(X_{uv}\cap M)Z/Z$ is almost simple with socle $\PSL_m(q)$, or one of the following holds.
\begin{itemize}
\item[(i)] $m=2$, $q=9$, and $\varphi_1(X_{uv}\cap M)Z/Z=\A_5$.
\item[(ii)] $m=2$, $q\geqslant7$ is a Mersenne prime, and $\varphi_1(X_{uv}\cap M)Z/Z\leqslant\C_q\rtimes\C_{q-1}$.
\item[(iii)] $m=2$, $q\geqslant4$ is even, and $\varphi_1(X_{uv}\cap M)Z/Z\leqslant\D_{2(q+1)}$.
\item[(iv)] $m=3$, $q=3$, and $\varphi_1(X_{uv}\cap M)Z/Z=\C_{13}\rtimes\C_3$.
\item[(v)] $m=4$, $q=2$, and $\varphi_1(X_{uv}\cap M)=\A_7$.
\item[(vi)] $m=6$, $q=2$, and $\varphi_1(X_{uv}\cap M)$ stabilizes a $1$-dimensional or $5$-dimensional subspace of $W_1$.
\end{itemize}

Assume that~(i) occurs. Then $|\varphi_1(X_{uv}\cap M)Z/Z|_3=3$, and so $|\varphi_1(X_{uv}\cap M)|_3=3$ as $|Z|=q-1=8$. Hence $|X_{uv}\cap M|_3\leqslant3^k$, which implies
\[
|X_{uv}|_3\leqslant|X_{uv}\cap M|_3|\Sy_k|_3\leqslant3^k(k!)_3.
\]
From~\eqref{eq4} we see that $|X_v|_3=3^{2k}(k!)_3$. Thus the valency of $\Gamma$ has $3$-part
\[
\frac{|X_v|_3}{|X_{uv}|_3}\geqslant\frac{3^{2k}(k!)_3}{3^k(k!)_3}=3^k.
\]
Since $\Gamma$ is $(G,3)$-arc-transitive, we conclude that $|G_v|$ is divisible by $3^{3k}$. However, as $|\Out(L)|_3=1$, we have $|G_v|_3=|L_v|_3=3^{2k}(k!)_3$. This leads to
\[
3^{3k}\leqslant|G_v|_3=3^{2k}(k!)_3,
\]
that is, $(k!)_3\geqslant3^k$, which is not possible.

Assume that~(ii) occurs. Then $|\varphi_1(X_{uv}\cap M)Z/Z|_2=2$, and so $|\varphi_1(X_{uv}\cap M)|_2=4$ as $|Z|_2=(q-1)_2=2$. Hence $|X_{uv}\cap M|_2\leqslant4^k$, which implies
\[
|X_{uv}|_2\leqslant|X_{uv}\cap M|_2|\Sy_k|_2\leqslant2^{2k}(k!)_2.
\]
From~\eqref{eq4} we see that $|X_v|_2=2^{2k-1}(q+1)^k(k!)_2$. Thus the valency of $\Gamma$ has $2$-part
\[
\frac{|X_v|_2}{|X_{uv}|_2}\geqslant\frac{2^{2k-1}(q+1)^k(k!)_2}{2^{2k}(k!)_2}=\frac{(q+1)^k}{2}.
\]
Since $\Gamma$ is $(G,3)$-arc-transitive, we conclude that $|G_v|$ is divisible by $(q+1)^{3k}/2^3$. However, as $|G_v|_2\leqslant|\Out(L)|_2|L_v|_2=2^{2k}(q+1)^k(k!)_2$, it follows that
\[
(q+1)^{3k}/2^3\leqslant|G_v|_2\leqslant2^{2k}(q+1)^k(k!)_2.
\]
This leads to $(q+1)^{2k}\leqslant2^{2k+3}(k!)_2$ and hence
\[
8^{2k}\leqslant(q+1)^{2k}\leqslant2^{2k+3}(k!)_2<2^{3k+3},
\]
a contradiction.

Assume that~(iii) occurs. Then $|\varphi_1(X_{uv}\cap M)Z/Z|_2=2$ and $|Z|_2=(q-1)_2=1$. Thus, $|\varphi_1(X_{uv}\cap M)|_2=2$ and so $|X_{uv}\cap M|_2\leqslant2^k$, which leads to
\[
|X_{uv}|_2\leqslant|X_{uv}\cap M|_2|\Sy_k|_2\leqslant2^k(k!)_2.
\]
From~\eqref{eq4} we see that $|X_v|_2=q^k(k!)_2$. Consequently, $q^k(k!)_2\leqslant(2^k(k!)_2)^2$ since $|X_v|$ divides $|X_{uv}|^2$. This implies that $q^k\leqslant2^{2k}(k!)_2<2^{3k}$ and hence $q=4$. As $\varphi_1(X_{uv}\cap M)Z/Z\leqslant\D_{10}$ and $|Z|=3$, we then derive that $|X_{uv}\cap M|_3\leqslant3^k$, whence
\[
|X_{uv}|_3\leqslant|X_{uv}\cap M|_3|\Sy_k|_3\leqslant3^k(k!)_3.
\]
By~\eqref{eq4} we have $|X_v|_3=3^{2k-1}(k!)_3$. Thus the valency of $\Gamma$ has $3$-part
\[
\frac{|X_v|_3}{|X_{uv}|_3}\geqslant\frac{3^{2k-1}(k!)_3}{3^k(k!)_3}=3^{k-1}.
\]
Since $\Gamma$ is $(G,3)$-arc-transitive, we conclude that $|G_v|$ is divisible by $3^{3(k-1)}$. This together with $|G_v|_3\leqslant|\Out(L)|_3|L_v|_3=3^{2k-1}(k!)_3<3^{2k-1}\cdot3^{k/2}$ implies
\[
3^{3(k-1)}<3^{2k-1}\cdot3^{k/2},
\]
which forces $k=2$ or $3$. However, for $q=4$ and $k=2$ or $3$, computation in \magma~\cite{magma} shows that there is no nontrivial homogeneous factorization of $X_v$ with the two factors conjugate in $X$, a contradiction.

Assume that~(iv) occurs. In this case, $|\varphi_1(X_{uv}\cap M)|_3=|\varphi_1(X_{uv}\cap M)Z/Z|_3=3$ as $|Z|_3=1$. Hence $|X_{uv}\cap M|_3\leqslant3^k$ and so
\[
|X_{uv}|_3\leqslant|X_{uv}\cap M|_3|\Sy_k|_3\leqslant3^k(k!)_3.
\]
From~\eqref{eq4} we see that $|X_v|_3=3^{3k}(k!)_3$. Then as $(k!)_3<3^k$, it follows that $|X_{uv}|_3^2\leqslant3^{2k}(k!)_3^2<|X_v|_3$. This implies that $|X_v|$ does not divide $|X_{uv}|^2$, a contradiction.

Assume that~(v) occurs. Then $|\varphi_1(X_{uv}\cap M)|_2=2^3$. Hence $|X_{uv}\cap M|_2\leqslant2^{3k}$ and so
\[
|X_{uv}|_2\leqslant|X_{uv}\cap M|_2|\Sy_k|_2\leqslant2^{3k}(k!)_2.
\]
From~\eqref{eq4} we see that $|X_v|_2=2^{6k}(k!)_2$. Thus the valency of $\Gamma$ has $2$-part
\[
\frac{|X_v|_2}{|X_{uv}|_2}\geqslant\frac{2^{6k}(k!)_2}{2^{3k}(k!)_2}=2^{3k}.
\]
Since $\Gamma$ is $(G,3)$-arc-transitive, we conclude that $|G_v|$ is divisible by $2^{9k}$. This together with $|G_v|_2\leqslant|\Out(L)|_2|L_v|_2=2^{6k+1}(k!)_2<2^{7k+1}$ implies $2^{9k}\leqslant|G_v|_2<2^{7k+1}$, which is not possible.

Next assume~(vi). In this case, $|\varphi_1(X_{uv}\cap M)|_7=|\varphi_1(X_{uv}\cap M)Z/Z|_7=7$ as $Z=1$. Hence $|X_{uv}\cap M|_7\leqslant7^k$ and so
\[
|X_{uv}|_7\leqslant|X_{uv}\cap M|_7|\Sy_k|_7\leqslant7^k(k!)_7.
\]
From~\eqref{eq4} we see that $|X_v|_7=7^{2k}(k!)_7$. Thus the valency of $\Gamma$ has $7$-part
\[
\frac{|X_v|_7}{|X_{uv}|_7}\geqslant\frac{7^{2k}(k!)_7}{7^k(k!)_7}=7^k.
\]
Since $\Gamma$ is $(G,3)$-arc-transitive, we conclude that $|G_v|$ is divisible by $7^{3k}$. However, as $|\Out(L)|_7=1$, we have $|G_v|_7=|L_v|_7=7^{2k}(k!)_7$. It follows that $7^{2k}(k!)_7\geqslant|G_v|_7\geqslant7^{3k}$, that is, $(k!)_7\geqslant7^k$, which is not possible.

Thus far we have seen that none of cases~(i)--(vi) is possible. As a consequence, $\varphi_1(X_{uv}\cap M)Z/Z$ is almost simple with socle $\PSL_m(q)$. Then since $\overline{X_{uv}}$ is transitive, it follows that $X_{uv}\cap M$ has a unique insoluble composition factor $\PSL_m(q)$ with multiplicity $\ell$ dividing $k$. We prove $\ell=k$ in the next paragraph.

Suppose to the contrary that $\ell<k$. Write $q=p^f$ with $p$ prime. First assume $(m,q)\neq(2,8)$. Then there exists an odd prime $r$ in $\pi(\PSL_m(q))\setminus\pi(q-1)$ such that $r>f$. It follows that $|\varphi_1(X_{uv}\cap M)|_r=|\varphi_1(X_{uv}\cap M)Z/Z|_r=|\PSL_m(q)|_r$ and $|\Out(L)|_r=1$. Since $\ell<k$, we deduce $|X_{uv}\cap M|_r\leqslant|\PSL_m(q)|_r^{k/2}$, so the valency of $\Gamma$ has $r$-part
\[
\frac{|X_v|_r}{|X_{uv}|_r}\geqslant\frac{|X_v|_r}{|X_{uv}\cap M|_r|\Sy_k|_r}\geqslant\frac{|\PSL_m(q)|^k(k!)_r}{|\PSL_m(q)|_r^{k/2}(k!)_r}=|\PSL_m(q)|_r^{k/2}.
\]
This implies $|G_v|_r\geqslant|\PSL_m(q)|_r^{3k/2}$ as $\Gamma$ is $(G,3)$-arc-transitive. However,
\[
|G_v|_r\leqslant|\Out(L)|_r|L_v|_r=|\PSL_m(q)|_r^k(k!)_r<|\PSL_m(q)|_r^kr^{k/(r-1)}\leqslant|\PSL_m(q)|_r^kr^{k/2}.
\]
We conclude that $|\PSL_m(q)|_r^{3k/2}<|\PSL_m(q)|_r^kr^{k/2}$ and hence
\[
r^{k/2}\leqslant|\PSL_m(q)|_r^{k/2}<r^{k/2},
\]
a contradiction. Next assume $(m,q)=(2,8)$. Then
\[
|\varphi_1(X_{uv}\cap M)|_3=|\varphi_1(X_{uv}\cap M)Z/Z|_3\leqslant|\PGL_2(8)|_3=9
 \]
and $|\Out(L)|_3=3$. Since $\ell<k$, we deduce $|X_{uv}\cap M|_3\leqslant9^{k/2}=3^k$, so the valency of $\Gamma$ has $3$-part
\[
\frac{|X_v|_3}{|X_{uv}|_3}\geqslant\frac{|X_v|_3}{|X_{uv}\cap M|_3|\Sy_k|_3}\geqslant\frac{9^k(k!)_3}{3^k(k!)_3}=3^k.
\]
This together with the $(G,3)$-arc-transitivity of $\Gamma$ implies $|G_v|_3\geqslant3^{3k}$, whence
\[
3^{3k}\leqslant|G_v|_3\leqslant|\Out(L)|_3|L_v|_3=3\cdot9^k(k!)_3<3\cdot3^{2k}\cdot3^{k/2},
\]
again a contradiction.

Now we have $\ell=k$. Accordingly, $X_{uv}\cap M\geqslant M'\cong\SL_m(q)^k$ and hence $M'$ is a normal subgroup of $X_{uv}$. Moreover,
\[
M'\Z(X_{uv})/\Z(X_{uv})\cong M'/(M'\cap\Z(X_{uv}))=M'/\Z(M')\cong\PSL_m(q)^k
\]
is a minimal normal subgroup of $X_{uv}/\Z(X_{uv})$ since $\overline{X_{uv}}$ is transitive. As $X_{vw}\cong X_{uv}$, we conclude that $X_{vw}$ has a normal subgroup $N$ isomorphic to $\SL_m(q)^k$ such that $N\Z(X_{uv})/\Z(X_{uv})\cong\PSL_m(q)^k$ is a minimal normal subgroup of $X_{vw}/\Z(X_{vw})$. Since $N\cap M$ is normal in $X_{vw}$, $(N\cap M)\Z(X_{vw})/\Z(X_{vw})$ is normal in $X_{vw}/\Z(X_{vw})$. Thus, $(N\cap M)\Z(X_{vw})/\Z(X_{vw})=1$ or $N\Z(X_{uv})/\Z(X_{uv})$ as $(N\cap M)\Z(X_{vw})/\Z(X_{vw})$ is a subgroup of $N\Z(X_{vw})/\Z(X_{vw})$. If $(N\cap M)\Z(X_{vw})/\Z(X_{vw})=1$, that is, $N\cap M\leqslant\Z(X_{vw})$, then the insoluble composition factors of $N$ coincide with all those of $N/(N\cap M)$ and so $|N/(N\cap M)|$ is divisible by $|\PSL_m(q)|^k$. Since $N/(N\cap M)\cong NM/M\leqslant\Sy_k$, this would imply that $k!$ is divisible by $|\PSL_m(q)|^k$, which is not possible. Therefore, $(N\cap M)\Z(X_{vw})/\Z(X_{vw})=N\Z(X_{uv})/\Z(X_{uv})$. It follows that
\[
N\cap M\geqslant(N\cap M)'=((N\cap M)Z(X_{vw}))'=(NZ(X_{vw}))'=N'=N,
\]
and so $N\leqslant M$. This implies $N=N'\leqslant M'$, which leads to $N=M'$ due to $|N|=|M'|$. In particular, $M'\Z(X_{uv})/\Z(X_{uv})$ is a minimal normal subgroup of $X_{vw}/\Z(X_{vw})$, and so $\overline{X_{vw}}$ is transitive. Since $X_{uv}\cong X_{vw}$ and both $X_{uv}$ and $X_{vw}$ contain $M'$, we see that $\overline{X_{uv}}$ and $\overline{X_{vw}}$ have the same insoluble composition factors. Thus we derive from Lemma~\ref{lem4} that both $\overline{X_{uv}}$ and $\overline{X_{uv}}$ contain $\A_k$. Consequently, $X_{uv}^{(\infty)}=X_{vw}^{(\infty)}=X_v^{(\infty)}$, and so
\[
(X_v^{(\infty)})^g=(X_{uv}^{(\infty)})^g=(X_{uv}^g)^{(\infty)}=X_{vw}^{(\infty)}=X_v^{(\infty)}.
\]
This in conjunction with the fact that $X_v^{(\infty)}$ is normal in $X_v$ implies that $X_v^{(\infty)}$ is normal in $\langle X_v,g\rangle=X$, a contradiction.
\end{proof}

We conclude this section with the following:

\begin{theorem}\label{thm2}
Let $\Gamma$ be a $G$-vertex-primitive $(G,s)$-arc-transitive digraph such that $G$ is almost simple with socle $\PSL_n(q)$ and $G_v$ is a maximal subgroup of $G$ from classes $\calC_1$ and $\calC_2$, where $v$ is a vertex of $\Gamma$. Then $s\leqslant2$. Moreover, if $G_v$ is from class $\calC_1$, then $G\nleqslant\PGaL_n(q)$ and $G_v$ does not stabilize a nontrivial proper subspace of $\bbF_q^n$.
\end{theorem}

\begin{proof}
From Lemma~\ref{lem5} we see that $\Gamma$ has valency at least $3$. Hence Hypothesis~$\ref{hyp1}$ holds. Then the theorem follows from Lemmas~\ref{lem16} and~\ref{lem18}.
\end{proof}

\section{$\calC_3$, $\calC_4$, $\calC_5$ and $\calC_6$-subgroups}\label{sec4}

We recall Hypothesis~\ref{hyp1} and observe that $G_v=G_{uv}G_{vw}$ with $G_{uv}\cong G_{vw}$ and so $\pi(G_v)=\pi(G_{uv})=\pi(G_{vw})$.

\begin{lemma}\label{lem23}
If Hypothesis~$\ref{hyp1}$ holds then $G_v$ is not a $\calC_3$-subgroup of $G$.
\end{lemma}

\begin{proof}
Suppose that Hypothesis~\ref{hyp1} holds and $G_v$ is a $\calC_3$-subgroup of $G$. Then by~\cite[Proposition~4.3.6]{KL1990}, either $n$ is prime, or $G_v/\Rad(G_v)$ is almost simple with socle $\PSL_{n/r}(q^r)$ for some prime divisor $r$ of $n$.

First assume that $n$ is a prime with $\ppd(p,nf)\neq\emptyset$. If $(n,q)=(2,8)$ then $G_v=\D_{18}$ or $\C_9\rtimes\C_6$, where $G=\PSL_2(8)$ or $\PGaL_2(8)$ respectively, but a \magma~\cite{magma} calculation shows that $G_v$ does not have a homogeneous factorization $G_v=G_{uv}G_{vw}$ with $|G_v|/|G_{uv}|\geqslant3$, a contradiction. Therefore, $(n,q)\neq(2,8)$. Then since
\[
\C_{(q^n-1)/(q-1)(q-1,n)}\rtimes\C_n\leqslant G_v\leqslant(\C_{(q^n-1)/(q-1)}\rtimes\C_{nf}).\C_2,
\]
we deduce that for any $m\in\ppd(p,nf)$ there is a unique subgroup $M$ of order $m$ in $G_v$. Since $m\in\pi(G_v)=\pi(G_{uv})=\pi(G_{vw})$, it follows that $M\leqslant G_{uv}$ and $M\leqslant G_{vw}$. Moreover, since $G_{uv}^{\overline{g}}=G_{vw}$ we have $M^{\overline{g}}=M$. However, this contradicts Lemma~\ref{lem26} as $M$ is normal in $G_v$.

Next assume that $n$ is a prime with $\ppd(p,nf)=\emptyset$. Then $q=p$ is a Mersenne prime and $n=2$. In this case $G_v=\D_{q+1}$ or $\D_{2(q+1)}$, where $G=\PSL_2(q)$ or $\PGL_2(q)$, respectively. Let $N$ be the unique cyclic subgroup of index $2$ of $G_v$. Then since $G_v=G_{uv}G_{vw}$, at least one of $G_{uv}$ or $G_{vw}$, say $G_{uv}$, is not contained in $N$. This implies that $G_{vw}$ is not contained in $N$ since $G_{vw}\cong G_{uv}$. Consequently, $G_{uv}\cap N$ and $G_{vw}\cap N$ are the unique cyclic subgroups of index $2$ of $G_{uv}$ and $G_{vw}$, respectively. Thus we conclude that $G_{uv}\cap N$ and $G_{vw}\cap N$ are subgroups of the cyclic group $N$ of the same order, and so $G_{uv}\cap N=G_{vw}\cap N$. Moreover, as $G_{vw}\cap N^{\overline{g}}=(G_{uv}\cap N)^{\overline{g}}\cong G_{uv}\cap N$ is a cyclic subgroup of index $2$ of $G_{vw}$, we deduce that $G_{vw}\cap N^{\overline{g}}=G_{vw}\cap N$ and hence $(G_{uv}\cap N)^{\overline{g}}=G_{vw}\cap N^{\overline{g}}=G_{vw}\cap N=G_{uv}\cap N$. Since $G_{uv}\cap N$ is characteristic in $N$ and hence normal in $G_v$, we have a contradiction to Lemma~\ref{lem26}.

Finally assume that $G_v/\Rad(G_v)$ is almost simple with socle $\PSL_{n/r}(q^r)$ for some prime divisor $r$ of $n$. If $(n/r,q^r)=(2,8)$ then $G_v=\GL_2(8)\rtimes\C_3$ or $\GL_2(8)\rtimes\C_6$, where $G=\PSL_6(2)$ or $\PSL_6(2).\C_2$ respectively, but a \magma~\cite{magma} calculation shows that $G_v$ does not have a homogeneous factorization $G_v=G_{uv}G_{vw}$ with $|G_v|/|G_{uv}|\geqslant3$, a contradiction. If $(n/r,q^r)=(2,9)$ then $\PSL_4(3)\leqslant G\leqslant\PSL_4(3).\C_2^2$ and $(\A_6\times\C_4)\rtimes\C_2\leqslant G_v\leqslant(\A_6\times\C_4).\C_2^3$ (see~\cite{atlas}), but by a \magma~\cite{magma} computation $G_v$ does not have a factorization $G_v=G_{uv}G_{vw}$ with $|G_v|/|G_{uv}|\geqslant3$ such that $G_{uv}$ and $G_{vw}$ are conjugate in $G$, still a contradiction. Therefore, $(n/r,q^r)\neq(2,8)$ or $(2,9)$. Denote $\overline{G_v}=G_v/\Rad(G_v)$, $\overline{G_{uv}}=G_{uv}\Rad(G_v)/\Rad(G_v)$ and $\overline{G_{vw}}=G_{vw}\Rad(G_v)/\Rad(G_v)$. Then since $G_v=G_{uv}G_{vw}$ and $\pi(\Rad(G_v))\subseteq\pi(q^r-1)\cup\pi(2rf)$ we have $\overline{G_v}=\overline{G_{uv}}\,\overline{G_{vw}}$ with $\pi(\overline{G_{uv}})$ and $\pi(\overline{G_{vw}})$ both containing $\pi(\overline{G_v})\setminus(\pi(q^r-1)\cup\pi(2rf))$. Hence we see from Lemma~\ref{lem19} that at least one of $\overline{G_{uv}}$ or $\overline{G_{vw}}$, say $\overline{G_{uv}}$, contains $\Soc(\overline{G_v})=\PSL_{n/r}(q^r)$. Since $G_{uv}\cong G_{vw}$ and $\Rad(G_v)$ is soluble, it follows that $\overline{G_{uv}}$ and $\overline{G_{vw}}$ have the same insoluble composition factors. Thus $\overline{G_{vw}}$ also contains $\Soc(\overline{G_v})=\PSL_{n/r}(q^r)$. Consequently, $G_{uv}$ and $G_{vw}$ both contain $G_v^{(\infty)}$, and so $G_{uv}^{(\infty)}=G_{vw}^{(\infty)}=G_v^{(\infty)}$. It follows that
\[
(G_v^{(\infty)})^{\overline{g}}=(G_{uv}^{(\infty)})^{\overline{g}}=(G_{uv}^{\overline{g}})^{(\infty)}=G_{vw}^{(\infty)}=G_v^{(\infty)},
\]
contradicting Lemma~\ref{lem26}.
\end{proof}

\begin{lemma}
If Hypothesis~$\ref{hyp1}$ holds then $G_v$ is not a $\calC_4$-subgroup of $G$.
\end{lemma}

\begin{proof}
Suppose that Hypothesis~\ref{hyp1} holds and $G_v$ is a $\calC_4$-subgroup of $G$. Then by~\cite[Proposition~4.4.10]{KL1990}, $G_v^{(\infty)}=\PSL_m(q)^{(\infty)}\times\PSL_k(q)$ with $1<m<k$ such that $n=mk$ and $\pi(G_v/G_v^{(\infty)})\subseteq\pi(\PSL_m(q))\cup\pi(f)$. Note that $G_v^{(\infty)}=\PSL_m(q)^{(\infty)}\times\PSL_k(q)$ has a unique normal subgroup $K\cong\PSL_k(q)$. We see that $K$ is characteristic in $G_v^{(\infty)}$ and thus normal in $G_v$. Let $C$ be the centralizer of $K$ in $G_v$. Then $C\lhd G_v$ and $G_v/C$ is an almost simple group with socle $\PSL_k(q)$. Denote $\overline{G_v}=G_v/C$, $\overline{G_{uv}}=G_{uv}C/C$ and $\overline{G_{vw}}=G_{vw}C/C$. Note that $C\cap G_v^{(\infty)}=\PSL_m(q)^{(\infty)}$ and so $\pi(C)\subseteq\pi(\PSL_m(q)^{(\infty)})\cup\pi(G_v/G_v^{(\infty)})$. We have
\begin{align*}
\pi(G_v)\setminus\pi(C)&\supseteq\pi(\PSL_k(q))\setminus(\pi(\PSL_m(q))\cup\pi(G_v/G_v^{(\infty)}))\\
&\supseteq\pi(q^k-1)\setminus(\pi(\PSL_m(q))\cup\pi(f)).
\end{align*}
Thus we deduce from $G_v=G_{uv}G_{vw}$ and $\pi(G_{uv})=\pi(G_{vw})=\pi(G_v)$ that $\overline{G_v}=\overline{G_{uv}}\,\overline{G_{vw}}$ with $\pi(\overline{G_{uv}})$ and $\pi(\overline{G_{vw}})$ both containing $\pi(q^k-1)\setminus(\pi(\PSL_m(q))\cup\pi(f))$. Then by Lemma~\ref{lem20}, one of the following holds:
\begin{itemize}
\item[(i)] at least one of $\overline{G_{uv}}$ or $\overline{G_{vw}}$ contains $\Soc(\overline{G_v})$;
\item[(ii)] $k=6$, $q=2$, and neither $\overline{G_{uv}}$ nor $\overline{G_{vw}}$ contains $\Soc(\overline{G_v})$.
\end{itemize}

First assume that~(i) occurs. Without loss of generality, assume that $\overline{G_{uv}}$ contains $\Soc(\overline{G_v})$. Then since $G_{uv}$ and $G_{vw}$ have the same composition factors, we see that $\overline{G_{vw}}$ also contains $\Soc(\overline{G_v})=\PSL_k(q)$. Consequently, $G_{uv}$ and $G_{vw}$ both contain $K$, and hence $K$ is the unique normal subgroup isomorphic to $\PSL_k(q)$ of $G_{uv}$ and $G_{vw}$, respectively. Now $K^{\overline{g}}$ is normal in $G_{uv}^{\overline{g}}=G_{vw}$ and so $K^{\overline{g}}=K$. This contradicts Lemma~\ref{lem26}.

Next assume that~(ii) occurs. Here by~\cite[Proposition~4.4.10]{KL1990}, $G_v=M\times K$ or $(M\times K).\C_2$ with $M=\PSL_m(2)$ for some $1<m<6$. Since $\overline{G_v}=\overline{G_{uv}}\,\overline{G_{vw}}$ with $\overline{G_v}=\PSL_6(2)$ or $\PSL_6(2).\C_2$, computation in \magma~\cite{magma} shows that, interchanging $G_{uv}$ and $G_{vw}$ if necessary, we have $\overline{G_{uv}}\leqslant\GaL_2(8).\C_2$, $\GaL_3(4).\C_2$ or $\Sp_6(2).\C_2$ and $\overline{G_{vw}}=\C_2^5\rtimes\PSL_5(2)$, $\PSL_5(2)$ or $\PSL_5(2).\C_2$. In particular, $\PSL_5(2)$ is a composition factor of $G_{vw}$. Thus $\PSL_5(2)$ is a composition factor of $G_{uv}$ as $G_{uv}\cong G_{vw}$. Since $\overline{G_{uv}}\leqslant\GaL_2(8)$, $\GaL_3(4)$ or $\Sp_6(2)$, we see that $\PSL_5(2)$ is not a composition factor of $\overline{G_{uv}}$. Therefore, $\PSL_5(2)$ is a composition factor of $G_{uv}\cap C$ and hence a composition factor of $G_{uv}\cap M$, which indicates that $m=5$ and $M\leqslant G_{uv}$. Now as $G_{uv}$ has a normal subgroup $M\cong\PSL_5(2)$, $G_{vw}$ also has a normal subgroup isomorphic to $\PSL_5(2)$, say $N$. Then $N\cap C=1$ or $N$, since $N$ is simple. If $N\cap C=N$, then $M=N=\PSL_5(2)$ and hence $G_{vw}$ has $\PSL_5(2)$ as a composition factor of multiplicity $2$. This would imply that $G_{uv}$ has $\PSL_5(2)$ as a composition factor of multiplicity $2$, which is not possible as $G_{uv}/C=\overline{G_{uv}}\leqslant\GaL_2(8).\C_2$, $\GaL_3(4).\C_2$ or $\Sp_6(2).\C_2$. Consequently, $N\cap C=1$, and so $\overline{G_{vw}}$ has a normal subgroup $NC/C\cong N$. Since $\C_2^5\rtimes\PSL_5(2)$ does not have a normal subgroup isomorphic to $\PSL_5(2)$, it follows that $\overline{G_{vw}}=\PSL_5(2)$ or $\PSL_5(2).\C_2$. Then searching in \magma~\cite{magma} for the factorization $\overline{G_v}=\overline{G_{uv}}\,\overline{G_{vw}}$ with $\PSL_6(2)\leqslant\overline{G_v}\leqslant\PSL_6(2).\C_2$ and $\PSL_5(2)\leqslant\overline{G_v}\leqslant\PSL_5(2).\C_2$ we deduce that $\overline{G_{uv}}$ has $\PSL_3(4)$, $\PSU_3(3)$ or $\Sp_6(2)$ as a composition factor. This implies that $G_{vw}$ has $\PSL_3(4)$, $\PSU_3(3)$ or $\Sp_6(2)$ as a composition factor, and so $G_{vw}\cap C$ has one of these groups as a composition factor since $\overline{G_{vw}}=\PSL_5(2)$. However, no subgroup of $C\leqslant\PSL_5(2).\C_2$ has $\PSL_3(4)$, $\PSU_3(3)$ or $\Sp_6(2)$ as a composition factor, a contradiction.
\end{proof}

\begin{lemma}\label{lem29}
If Hypothesis~$\ref{hyp1}$ holds then $G_v$ is not a $\calC_5$-subgroup of $G$.
\end{lemma}

\begin{proof}
Suppose that Hypothesis~\ref{hyp1} holds and $G_v$ is a $\calC_5$-subgroup of $G$. Then there exists a prime $r$ such that $q^{1/r}$ is a power of $p$ and $L_v$ is described in~\cite[Proposition~4.5.3]{KL1990}. If $n=2$ and $q^{1/r}=2$, then $q=4$ by~\cite[Table~8.1]{BHR2013}. However, in this case $G_v=\Sy_3$ or $\Sy_3\times\C_2$, which have no factorization $G_v=G_{uv}G_{vw}$ with $|G_v|/|G_{uv}|\geqslant2$ such that $G_{uv}$ and $G_{vw}$ are conjugate in $G$. If $(n,q^{1/r})=(2,3)$ then $\A_4\leqslant G_v\leqslant\Sy_4\times\C_r$ since $\PGL_2(3)\cong\Sy_4$. However, in this case $G_v$ does not have an appropriate factorization $G_v=G_{uv}G_{vw}$. Thus $(n,q^{1/r})\neq(2,2)$ or $(2,3)$. Hence $\G_v^{(\infty)}=\PSL_n(q^{1/r})$. Then we see from~\cite[Proposition~4.5.3]{KL1990} that $N:=\Z(G_v)$ has order $1$ or $r$, and $G_v/N$ is an almost simple group with socle $\PSL_n(q^{1/r})$.

For any subgroup $H$ of $G_v$ denote $\overline{H}=HN/N$. Since $G_v=G_{uv}G_{vw}$ we have $\overline{G_v}=\overline{G_{uv}}\,\overline{G_{vw}}$, whence $|\overline{G_v}|$ divides $|\overline{G_{uv}}||\overline{G_{vw}}|$. If at least one of $G_{uv}$ or $G_{vw}$, say $G_{vw}$, does not contain $N$, then $G_{vw}\cap N=1$ and so $|\overline{G_{vw}}|=|G_{vw}|=|G_{uv}|$ is divisible by $|\overline{G_{uv}}|$, which implies that $|\overline{G_v}|$ divides $|\overline{G_{vw}}|^2$. If both $G_{uv}$ and $G_{vw}$ contain $N$, then $|\overline{G_{uv}}|=|G_{uv}|/r=|G_{vw}|/r=|\overline{G_{vw}}|$ and so $|\overline{G_v}|$ divides $|\overline{G_{uv}}|^2=|\overline{G_{vw}}|^2$. Therefore, $|\overline{G_v}|$ always divides at least one of $|\overline{G_{uv}}|^2$ or $|\overline{G_{vw}}|^2$, say $|\overline{G_{vw}}|^2$. Thus by~\cite[Corollary~5]{LPS2000} the pair $(\overline{G_v},\overline{G_{vw}})$ is described in~\cite[Table~10.7]{LPS2000}. Checking the condition that $|\overline{G_v}|$ divides $|\overline{G_{vw}}|^2$ for the candidates we obtain one of the following:
\begin{itemize}
\item[(i)] $\overline{G_{vw}}\geqslant\Soc(\overline{G_v})$;
\item[(ii)] $n=2$, $q^{1/r}=9$ and $\overline{G_{vw}}\cap\Soc(\overline{G_v})=\A_5$;
\item[(iii)] $n=4$, $q^{1/r}=2$ and $\overline{G_{vw}}\cap\Soc(\overline{G_v})=\A_7$;
\item[(iv)] $n=6$, $q^{1/r}=2$ and $\overline{G_{vw}}\cap\Soc(\overline{G_v})=\PSL_2(5)$ or $\C_2^5\rtimes\PSL_5(2)$.
\end{itemize}

First assume that~(i) occurs. Then $G_{vw}\geqslant G_v^{(\infty)}=\PSL_n(q^{1/r})$. Since $G_{uv}\cong G_{vw}$, it follows that $G_{uv}\geqslant G_v^{(\infty)}$ and so $G_{uv}^{(\infty)}=G_v^{(\infty)}=G_{vw}^{(\infty)}$. Since $(G_{uv}^{(\infty)})^{\overline{g}}=G_{vw}^{(\infty)}$, this implies that ${\overline{g}}$ normalizes $G_v^{(\infty)}$, contradicting Lemma~\ref{lem26}.

Next assume that~(ii) occurs. Here $L=\PSL_2(9^r)$, and $L_v=\PGL_2(9)$ if $r=2$ and $\PSL_2(9)$ if $r>2$. If $r=2$, then $G_v$ does not have a nontrivial homogeneous factorization, a contradiction. Therefore, $r$ is an odd prime, $L_v=\PSL_2(9)$ and $L_{uv}\cong L_{vw}\cong\A_5$. By~\cite[Table~8.1]{BHR2013}, $L$ has exactly one conjugacy class of subgroups isomorphic to $L_v$, and $\Nor_L(L_v)=L_v$. Note that $3^r\equiv\pm2\pmod{5}$ as $r$ is odd. Thus we see that $|\PGL_2(3^r)|$ is not divisible by $5$ and hence not divisible by $|\Nor_L(L_{vw})|$. Then~\cite[Table~8.1]{BHR2013} implies that the only maximal subgroups of $L$ that may contain $\Nor_L(L_{vw})$ are those isomorphic to $\PSL_2(9)$. Hence $\Nor_L(L_{vw})=L_{vw}$. Let $m$ be the number of subgroups of $G$ isomorphic to $L_v$ that contain $L_{vw}$. Note that $L_v=\PSL_2(9)$ has $12$ distinct subgroups isomorphic to $\A_5$, and there are exactly two conjugacy classes of such subgroups in $L$ (see for example~\cite[Exercise~2,~Page~416]{Suzuki1982}). By counting the number of pairs $(N_1,N_2)$ of subgroups of $L$ such that $N_1$ is isomorphic to $L_v$ and $N_1>N_2\cong L_{vw}$, one obtains
\[
\frac{|L|}{|\Nor_L(L_v)|}\cdot12=2\cdot\frac{|L|}{|\Nor_L(L_{vw})|}\cdot m.
\]
Accordingly, $L_{vw}$ is contained in exactly
\[
m=\frac{12|\Nor_L(L_{vw})|}{2|\Nor_L(L_v)|}=\frac{12|L_{vw}|}{2|L_v|}=1
\]
subgroup of $L$ that is isomorphic to $L_v$. Since $L_{vw}$ is contained in both $L_v$ and $L_w=L_v^{\overline{g}}$, we conclude that $L_v=L_w=L_v^{\overline{g}}$, contradicting Lemma~\ref{lem26}.

Finally assume that~(iii) or~(iv) occurs. Then we have seen above that either $|\overline{G_{vw}}|=|G_{vw}|=|G_{uv}|$ or $|\overline{G_{uv}}|=|\overline{G_{vw}}|$. For the former, $|\overline{G_{vw}}|/|\overline{G_{uv}}|=|G_{uv}|/|\overline{G_{uv}}|$ is a divisor of $|N|$. Thus we always have $|\overline{G_{vw}}|/|\overline{G_{uv}}|=1$ or $r$. However, a \magma~\cite{magma} calculation shows that there is no factorization $\overline{G_v}=\overline{G_{uv}}\,\overline{G_{vw}}$ with $(\Soc(\overline{G_v}),\overline{G_{vw}}\cap\Soc(\overline{G_v}))=(\A_8,\A_7)$, $(\PSL_6(2),\PSL_2(5))$ or $(\PSL_6(2),\C_2^5\rtimes\PSL_5(2))$ and $|\overline{G_{vw}}|/|\overline{G_{uv}}|$ being $1$ or a prime. This contradiction completes the proof.
\end{proof}

\begin{lemma}\label{lem25}
Let $X=\Sp_{2m}(r)$ with $r$ prime and $r^m\geqslant5$. Then $X$ does not have a subgroup of index $d$ such that $r^{2m}$ divides $2(r-1)d$ and $d$ divides $2(r-1)r^{2m}$.
\end{lemma}

\begin{proof}
Suppose for a contradiction that $X$ has a subgroup $Y$ of index $d$ such that $r^{2m}$ divides $2(r-1)d$ and $d$ divides $2(r-1)r^{2m}$. If $r=2$, then $X\cong\PSp_{2m}(2)$ with $m\geqslant3$ and $d$ is a power of $2$ such that $2^{2m-1}\leqslant d\leqslant2^{2m+1}$, contradicting~\cite{Guralnick1983}. Thus, $r\geqslant3$ and so $r^{2m}$ divides $d$. As $|\Sp_2(r)|_r=r$, this implies that $m\neq1$.

First assume that $m=2$. Let $M$ be a maximal subgroup of $X=\Sp_4(r)$ containing $Y$. Then $|X|/|M|$ divides $2(r-1)r^4$. Checking~\cite[Table~8.12]{BHR2013} we deduce that $r=3$ and $M=\Sp_2(9)\rtimes\C_2$ or $2_-^{1+4}.\A_5$. However, such an $M$ does not have a subgroup $Y$ such that $3^4$ divides $d$ and $d$ divides $2^2\cdot3^4$, a contradiction.

Next assume that $m\geqslant3$. Let $Z$ be the center of $X$, $\overline{X}=X/Z$ and $\overline{Y}=YZ/Z$. Then $|\overline{X}|/|\overline{Y}|$ divides $2(r-1)r^{2m}$, and thus from~\cite[Theorem~4]{LPS2000} we infer that $\overline{X}=\overline{Y}$. However, this implies that $d=|X|/|Y|$ divides $|X|/|\overline{Y}|=|X|/|\overline{X}|=|Z|$, contradicting the condition that $r^{2m}$ divides $d$.
\end{proof}

\begin{lemma}\label{lem17}
If Hypothesis~$\ref{hyp1}$ holds then $G_v$ is not a $\calC_6$-subgroup of $G$.
\end{lemma}

\begin{proof}
Suppose that Hypothesis~\ref{hyp1} holds and $G_v$ is a $\calC_6$-subgroup of $G$. If $n\leqslant4$ then the possibilities for $G_v$ can be obtained from~\cite[Tables~8.1,~8.3~and~8.8]{BHR2013} and in all these cases, a computation in \magma~\cite{magma} shows that $G_v$ has no homogeneous factorization $G_v=G_{uv}G_{vw}$ such that $|G_v|/|G_{uv}|\geqslant3$. Thus $n\geqslant5$, and so by~\cite[Propositions~4.6.5--4.6.6]{KL1990}, $n=r^m$ for some prime $r\neq p$ and $L_v=\C_r^{2m}.\Sp_{2m}(r)$. Moreover, since $G_v$ is maximal in $G$, it can be read off from~\cite[Tables~8.18,~8.35,~8.44~and~8.54]{BHR2013} and \cite[Table~3.5.A]{KL1990} that $G/L\leqslant\C_2\times\C_f$ and $f$ is the smallest odd integer such that $r\gcd(2,r)$ divides $p^f-1$. Consequently, $f$ divides $r-1$ and $|G/L|$ divides $2(r-1)$. Therefore, $G_v/\Rad(G_v)$ is an almost simple group with socle $\PSp_{2m}(r)$ and
\[
\pi(\Rad(G_v))\subseteq\pi(\Rad(L_v))\cup\pi(G/L)\subseteq\{2,r\}\cup\pi(2(r-1))=\pi(r(r-1)).
\]
For any subgroup $H$ of $G_v$ denote $\overline{H}=H\Rad(G_v)/\Rad(G_v)$. Then $\overline{G_v}=\overline{G_{uv}}\,\overline{G_{vw}}$ with $\pi(\overline{G_{uv}})$ and $\pi(\overline{G_{vw}})$ both containing $\pi(\overline{G_v})\setminus\pi(r(r-1))$. Hence we deduce from Lemma~\ref{lem24} that
one of the following holds:
\begin{itemize}
\item[(i)] at least one of $\overline{G_{uv}}$ or $\overline{G_{vw}}$ contains $\Soc(\overline{G_v})$;
\item[(ii)] $m=1$, $r$ is a Mersenne prime, and at least one of $\overline{G_{uv}}$ or $\overline{G_{vw}}$ has intersection with $\Soc(\overline{G_v})$ of odd order.
\end{itemize}

First assume case~(i). Without loss of generality, assume that $\overline{G_{uv}}$ contains $\Soc(\overline{G_v})=\PSp_{2m}(r)$. Note that $G_{uv}\cong G_{vw}$, and $G_{uv}/L_{uv}$ and $G_{vw}/L_{vw}$ are both soluble. Hence $\overline{L_{uv}}$ and $\overline{L_{vw}}$ both have $\PSp_{2m}(r)$ as an insoluble composition factor. It follows that
\[
L_{uv}\bfO_r(L_v)/\bfO_r(L_v)=L_{vw}\bfO_r(L_v)/\bfO_r(L_v)=L_v/\bfO_r(L_v)=\Sp_{2m}(r).
\]
Since $L_v/\bfO_r(L_v)=\Sp_{2m}(r)$ is irreducible on $\bfO_r(L_v)=\C_r^{2m}$, we deduce that $L_{uv}\cap\bfO_r(L_v)=1$ or $\C_r^{2m}$. If $L_{uv}\cap\bfO_r(L_v)=\C_r^{2m}$ then $|L_{uv}|=|L_v|$, which is not possible as $\Gamma$ has valency at least $2$. Consequently, $L_{uv}\cap\bfO_r(L_v)=1$. Similarly, $L_{vw}\cap\bfO_r(L_v)=1$ and so $L_{uv}\cong L_{vw}\cong\Sp_{2m}(r)$. Note that $|G_v|/|G_{uv}|=|G_{uv}|/|G_{uvw}|$ as $\Gamma$ is $(G,2)$-arc-transitive. Since $|L_v|/|L_{uv}|$ divides $|G_v|/|G_{uv}|$, we conclude that $r^{2m}$ divides $|G_v|/|G_{uv}|=|G_{uv}|/|G_{uvw}|$ and so divides $|G/L||L_{uv}|/|L_{uvw}|$. This implies that $r^{2m}$ divides $2(r-1)|L_{uv}|/|L_{uvw}|$. Moreover, $|L_{uv}|/|L_{uvw}|$ divides $|G_{uv}|/|G_{uvw}|=|G_v|/|G_{uv}|$ and so divides $|G/L||L_v|/|L_{uv}|$, which implies that $|L_{uv}|/|L_{uvw}|$ divides $2(r-1)r^{2m}$. Thus $L_{uvw}$ is isomorphic to a subgroup of $\Sp_{2m}(r)$ of index $d$ such that $r^{2m}$ divides $2(r-1)d$ and $d$ divides $2(r-1)r^{2m}$, which is not possible by Lemma~\ref{lem25}.

Next assume case~(ii). Without loss of generality, assume that $|\overline{G_{uv}}\cap\Soc(\overline{G_v})|$ is odd. Then $|L_{uv}|_2\leqslant2$ and so $|G_{uv}|_2\leqslant2|G/L|_2$. Since $G_v=G_{uv}G_{vw}$ with $G_{uv}\cong G_{vw}$, we derive that $|G_v|$ divides $|G_{uv}|^2$. Thus $|L_v|_2|G/L|_2=|G_v|_2$ divides $(2|G/L|_2)^2$, which implies that $2(r+1)=|L_v|_2$ divides $4|G/L|_2$. However, as $|G/L|$ divides $2f$ and $f$ is odd, we have $|G/L|_2\leqslant2$. This leads to $2(r+1)\leqslant8$ and so $r\leqslant3$, contrary to the condition that $r=r^m=n\geqslant5$.
\end{proof}

We conclude this section with the following:

\begin{theorem}\label{thm3}
Let $G$ be an almost simple group with socle $\PSL_n(q)$. Then there is no $G$-vertex-primitive $(G,2)$-arc-transitive digraph $\Gamma$ such that $G_v$ is a maximal subgroup of $G$ from classes $\calC_3$--$\calC_6$ for any vertex $v$ of $\Gamma$.
\end{theorem}

\begin{proof}
Suppose that $\Gamma$ is a $G$-vertex-primitive $(G,2)$-arc-transitive digraph such that $G_v$ is a maximal subgroup of $G$ from classes $\calC_3$--$\calC_6$, where $v$ is a vertex of $\Gamma$. Then by Lemma~\ref{lem5}, $\Gamma$ has valency at least $3$. Hence Hypothesis~\ref{hyp1} holds. According to Lemmas~\ref{lem23}--\ref{lem29} and~\ref{lem17}, $G_v$ cannot be a $\calC_i$-subgroup of $G$ for $3\leqslant i\leqslant6$, a contradiction.
\end{proof}

\section{$\calC_7$ and $\calC_8$-subgroups}\label{sec2}

In this section we need to consider the stronger hypothesis that $\Gamma$ is $(G,3)$-arc-transitive so that we only need to consider the structure of $L_v$ instead of $G_v$.

\begin{hypothesis}\label{hyp2}
Let $\Gamma$ be a $G$-vertex-primitive $(G,3)$-arc-transitive digraph of valency at least $3$, where $G$ is almost simple with socle $L=\PSL_n(q)$ and $q=p^f$ for some prime $p$. Then by Lemma~\ref{lem7}, $\Gamma$ is $(L,2)$-arc-transitive. Take an arc $u\rightarrow v$ of $\Gamma$. Let $g$ to be an element of $L$ such that $u^g=v$ and let $w=v^g$. Then $u\rightarrow v\rightarrow w\rightarrow w^g$ is a $3$-arc in $\Gamma$.
\end{hypothesis}

Under Hypothesis~\ref{hyp2}, it follows from Lemma~\ref{lem15} that $L_v=L_{uv}L_{vw}$ and $G_{vw}=G_{uvw}G_{uvw}^g$. Moreover, these are homogeneous factorizations. Hence by Lemma~\ref{lem2}, $\pi(L_{uv})=\pi(L_{vw})=\pi(L_v)$ and $\pi(G_{uvw})=\pi(G_{vw})$.


\begin{lemma}\label{lem27}
If Hypothesis~$\ref{hyp2}$ holds then $G_v$ is not a $\calC_7$-subgroup of $G$.
\end{lemma}

\begin{proof}
Suppose that Hypothesis~\ref{hyp2} holds and $G_v$ is a $\calC_7$-subgroup of $G$. Then by~\cite[Proposition~4.7.3]{KL1990},
\[
L_v\leqslant(M_1\times\cdots\times M_k)\rtimes\Sy_k=\PGL_m(q)\wr\Sy_k
\]
with $M_1\cong\cdots\cong M_k\cong\PGL_m(q)$, where $n=m^k$ with $m\geqslant3$, and
\[
L_v\geqslant(\Soc(M_1)\times\dots\times\Soc(M_k))\rtimes\Sy_k=\PSL_m(q)\wr\Sy_k.
\]
Let $M=L_v\cap(M_1\times\dots\times M_k)$ and for each $i=1,\dots,k$ let $\varphi_i$ be the projection of $M$ to $M_i$. Denote $\overline{L_v}=L_vM/M$, $\overline{L_{uv}}=L_{uv}M/M$ and $\overline{L_{vw}}=L_{vw}M/M$. From the factorization $L_v=L_{uv}L_{vw}$ we deduce that $\overline{L_v}=\overline{L_{uv}}\,\overline{L_{vw}}$. Then by Lemma~\ref{lem1}, at least one of $\overline{L_{uv}}$ or $\overline{L_{vw}}$, say $\overline{L_{uv}}$, is a transitive subgroup of $\Sy_k$. It follows from Lemma~\ref{lem3} that
\[
\varphi_1(L_{uv}\cap M)\cong\cdots\cong\varphi_k(L_{uv}\cap M)
\]
and $\pi(\PSL_m(q))\subseteq\pi(\varphi_1(L_{uv}\cap M))$. Thereby we deduce from~\cite[Corollary~5]{LPS2000} that either $\PSL_m(q)\leqslant\varphi_1(L_{uv}\cap M)\leqslant\PGL_m(q)$, or one of the following holds.
\begin{itemize}
\item[(i)] $m=4$, $q=2$, and $\varphi_1(L_{uv}\cap M)=\A_7$;
\item[(ii)] $m=6$, $q=2$, and $|\PSL_m(q)|/|\varphi_1(L_{uv}\cap M)|$ is divisible by $63$.
\end{itemize}

First assume that~(i) occurs. Then we have $|\varphi_1(L_{uv}\cap M)|_2=2^3$ and hence $|L_{uv}|_2\leqslant|L_{uv}\cap M|_2^k(k!)_2\leqslant2^{3k}(k!)_2$. Moreover, $|L_v|_2=|\PSL_4(2)\wr\Sy_k|_2=2^{6k}(k!)_2$. Thus the valency of $\Gamma$ has $2$-part
\[
\frac{|L_v|_2}{|L_{uv}|_2}\geqslant\frac{2^{6k}(k!)_2}{2^{3k}(k!)_2}=2^{3k}.
\]
Since $\Gamma$ is $(G,3)$-arc-transitive, we conclude that $|G_v|$ is divisible by $2^{9k}$. This together with $|G_v|_2\leqslant|L_v|_2|\Out(L)|_2=2^{6k+1}(k!)_2<2^{7k+1}$ implies $2^{9k}<2^{7k+1}$, which is not possible.

Next assume~(ii) occurs. Then $|\PSL_m(q)|_7/|\varphi_1(L_{uv}\cap M)|_7\geqslant7$, and so the valency of $\Gamma$ has $7$-part
\[
\frac{|L_v|_7}{|L_{uv}|_7}\geqslant\frac{|M|_7}{|L_{uv}\cap M|_7}\geqslant\left(\frac{|\PSL_m(q)|_7}{|\varphi_1(L_{uv}\cap M)|_7}\right)^k\geqslant7^k.
\]
Since $\Gamma$ is $(G,3)$-arc-transitive, we conclude that $|G_v|$ is divisible by $7^{3k}$. However, as $|\Out(L)|_7=1$, we have $|G_v|_7=|L_v|_7=|\PSL_6(2)|_7^k(k!)_7=7^{2k}(k!)_7$. It follows that $7^{2k}(k!)_7\geqslant7^{3k}$, that is, $(k!)_7\geqslant7^k$, which is not possible.

Thus far we have seen that neither case~(i) nor~(ii) is possible. Thus $\PSL_m(q)\leqslant\varphi_1(L_{uv}\cap M)\leqslant\PGL_m(q)$. Write $q=p^f$ with $p$ prime. Then there exists an odd prime $r$ in $\pi(\PSL_m(q))\setminus\pi(q-1)$ such that $r>f$. It follows that $|\varphi_1(L_{uv}\cap M)|_r=|\PSL_m(q)|_r$ and $|\Out(L)|_r=1$. Since $\overline{L_{uv}}$ is transitive and $\varphi_1(L_{uv}\cap M)$ has socle $\PSL_m(q)$, $L_{uv}\cap M$ has a unique insoluble composition factor $\PSL_m(q)$, and it has multiplicity $\ell$ dividing $k$. If $\ell<k$, then $|L_{uv}\cap M|_r\leqslant|\PSL_m(q)|_r^{k/2}$, and so the valency of $\Gamma$ has $r$-part
\[
\frac{|L_v|_r}{|L_{uv}|_r}\geqslant\frac{|\PSL_m(q)|^k(k!)_r}{|\PSL_m(q)|_r^{k/2}(k!)_r}=|\PSL_m(q)|_r^{k/2}.
\]
This together with the $(G,3)$-arc-transitivity of $\Gamma$ implies $|G_v|_r\geqslant|\PSL_m(q)|_r^{3k/2}$, which is not possible since
\[
|G_v|_r\leqslant|\Out(L)|_r|L_v|_r=|\PSL_m(q)|_r^k(k!)_r<|\PSL_m(q)|_r^kr^{k/(r-1)}\leqslant|\PSL_m(q)|_r^kr^{k/2}.
\]
Hence we have $\ell=k$ and so $L_{uv}\geqslant M'\cong\PSL_m(q)^k$. Moreover, since $\overline{L_{uv}}$ is transitive, $M'$ is a minimal normal subgroup of $L_{uv}$. As $L_{vw}\cong L_{uv}$, we conclude that $L_{vw}$ has a minimal normal subgroup $N$ isomorphic to $\PSL_m(q)^k$. Then since $N\cap M$ is normal in $L_{vw}$, either $N\cap M=1$ or $N\leqslant M$. If $N\cap M=1$, then $\PSL_m(q)^k\cong N\cong NM/M\leqslant\Sy_k$, which is not possible. Hence $N\leqslant M$, and it follows that $N=N'\leqslant M'$. This leads to $N=M'$ since $|N|=|M'|$. Therefore, $M'=N$ is a minimal normal subgroup of $L_{vw}$, which implies that $\overline{L_{vw}}$ is transitive. Since $L_{uv}\cong L_{vw}$ and both $L_{uv}$ and $L_{vw}$ contain $M'$, we see that $\overline{L_{uv}}$ and $\overline{L_{vw}}$ have the same insoluble composition factors. Thus we derive from Lemma~\ref{lem4} that both $\overline{L_{uv}}$ and $\overline{L_{uv}}$ contain $\A_k$. Consequently, $L_{uv}^{(\infty)}=L_{vw}^{(\infty)}=L_v^{(\infty)}$, and so
\[
(L_v^{(\infty)})^g=(L_{uv}^{(\infty)})^g=(L_{uv}^g)^{(\infty)}=L_{vw}^{(\infty)}=L_v^{(\infty)},
\]
contradicting Lemma~\ref{lem26}.
\end{proof}

\begin{lemma}\label{lem6}
If Hypothesis~$\ref{hyp2}$ holds with $n=4$ and $q$ odd, then $L_v$ is not the $\calC_8$-subgroup $\PSO_4^+(q).\C_2$ of $L$.
\end{lemma}

\begin{proof}
Suppose that Hypothesis~\ref{hyp2} holds with $n=4$ and $q$ odd while $L_v=\PSO_4^+(q).\C_2$ is a $\calC_8$-subgroup of $L$. For $q=3$, $5$, $7$ or $9$ computation in \magma~\cite{magma} shows that there is no nontrivial homogeneous factorization of $L_v$ with the two factors conjugate in $L$. Therefore $q\geqslant11$. Let $M=K_1\times K_2$ be the normal subgroup of $L_v$ of index $4$ such that $K_1\cong K_2\cong\PSL_2(q)$, and $\varphi_i$ be the projection of $M$ onto $K_i$ for $i=1,2$. Write $q=p^f$ with $p$ prime. Note that $|G_v|$ divides $|L_v||\Out(L)|=8f\gcd(q-1,4)|\PSL_2(q)|^2$.

Suppose $\varphi_i(L_{uv}\cap M)\neq K_i$ for $i=1$ or $2$. Then there is a maximal subgroup $H$ of $K_i\cong\PSL_2(q)$ containing $\varphi_i(L_{uv}\cap M)$. Since $\Gamma$ is $(G,3)$-arc-transitive and $|L_v|/|L_{uv}|$ is divisible by $|M|/|L_{uv}\cap M|$, we derive that $|G_v|$ is divisible by $(|M|/|L_{uv}\cap M|)^3$. Then since $|M|/|L_{uv}\cap M|$ is divisible by $|\varphi_i(M)|/|\varphi_i(L_{uv}\cap M)|$ and hence by $|\PSL_2(q)|/|H|$, it follows that $|G_v|$ is divisible by $(|\PSL_2(q)|/|H|)^3$. Consequently, $(|\PSL_2(q)|/|H|)^3$ divides $8f\gcd(q-1,4)|\PSL_2(q)|^2$. However, checking for all the possible maximal subgroups (see for example~\cite{Huppert1967}) $H$ of $\PSL_2(q)$ we see that this condition is not satisfied as $q\geqslant11$, a contradiction.

Now we have $\varphi_i(L_{uv}\cap M)=K_i$ for $i=1,2$, which implies that either $L_{uv}\cap M=M$ or $L_{uv}\cap M\cong\PSL_2(q)$. For the latter, $|L_v|/|L_{uv}|$ is divisible by $|M|/|L_{uv}\cap M|=|\PSL_2(q)|$, and so the $(G,3)$-arc-transitivity of $\Gamma$ implies that $|\PSL_2(q)|^3$ divides $8f\gcd(q-1,4)|\PSL_2(q)|^2$, which is not possible. Thus $L_{uv}\cap M=M$, from which we conclude $L_{uv}'=M=L_v'$. For the same reason, $L_{vw}'=M=L_v'$. Consequently,
\[
(L_v')^g=(L_{uv}')^g=(L_{uv}^g)'=L_{vw}'=L_v',
\]
contradicting Lemma~\ref{lem26}.
\end{proof}

\begin{lemma}\label{lem28}
If Hypothesis~$\ref{hyp2}$ holds then $G_v$ is not a $\calC_8$-subgroup of $G$.
\end{lemma}

\begin{proof}
Suppose that Hypothesis~\ref{hyp2} holds and $G_v$ is a $\calC_8$-subgroup of $G$. Then by~\cite[Propositions~4.8.3--4.8.5]{KL1990}, one of the following holds:
\begin{itemize}
\item[(i)] $n\geqslant4$ is even, and $L_v=\PSp_n(q).\C_{\gcd(q-1,2)\gcd(q-1,n/2)/\gcd(q-1,n)}$;
\item[(ii)] $n\geqslant3$, $q$ is odd, and $L_v=\PSO_n^\varepsilon(q).\C_{\gcd(n,2)}$ with $\varepsilon\in\{0,\pm\}$;
\item[(iii)] $n\geqslant3$, $q$ is a square, and $\Soc(L_v)=\PSU_n(q^{1/2})$.
\end{itemize}

First assume that~(i) occurs. Then by Proposition~\ref{Homogeneous}, we deduce from the homogeneous factorization $L_v=L_{uv}L_{vw}$ and the condition $|L_v|/|L_{uv}|\geqslant3$ that $n=4$, $q$ is even and $L_{uv}\cong\SL_2(q^2).\C_a$ with $a\in\{1,2\}$. It follows that $\SL_2(q^2)\leqslant G_{vw}\leqslant\GaL_2(q^2)\times\C_2$ and
\[
\frac{|G_{vw}|}{|G_{uvw}|}=\frac{|L_v|}{|L_{uv}|}=\frac{|\Sp_4(q)|}{|\SL_2(q^2).\C_a|}=\frac{q^2(q^2-1)}{a}.
\]
This implies that $\pi(G_{uvw})\neq\pi(G_{vw})$, a contradiction.

Next assume that~(ii) occurs. Then by Proposition~\ref{Homogeneous}, we deduce from the homogeneous factorization $L_v=L_{uv}L_{vw}$ that either $\Soc(L_{uv})=\Soc(L_{vw})=\Soc(L_v)$, or $(n,q,L_v)$ lies in Table~\ref{tab2} below. If $\Soc(L_{uv})=\Soc(L_{vw})=\Soc(L_v)$ is nonabelian simple, then
\[
(\Soc(L_v))^g=(\Soc(L_{uv}))^g=\Soc(L_{uv}^g)=\Soc(L_{vw})=\Soc(L_v),
\]
contradicting Lemma~\ref{lem26}. Now we analyze the candidates in Table~\ref{tab2}. For row~1 and row~2 of Table~\ref{tab2}, $L_v$ does not have a homogeneous factorization $L_v=L_{uv}L_{vw}$ with $|L_v|/|L_{uv}|\geqslant3$, a contradiction. For row~3 of Table~\ref{tab2}, $\Gamma$ has valency $|L_v|/|L_{uv}|=6$ since $L_v=L_{uv}L_{vw}$ is a homogeneous factorization with $|L_v|/|L_{uv}|\geqslant3$, but $|G_v|$ is not divisible by $6^3$, contrary to the $(G,3)$-arc-transitivity of $\Gamma$. By Lemma~\ref{lem6} we know that row~4 of Table~\ref{tab2} is not possible. For row~5 of Table~\ref{tab2}, we have $|L_{uv}|=a|\Omega_7(q)|$ with $a$ dividing $4$ and so
\[
\frac{|G_{vw}|}{|G_{uvw}|}=\frac{|L_v|}{|L_{uv}|}=\frac{|\PSO_4^+(q).\C_2|}{a|\Omega_7(q)|}=\frac{q^3(q^4-1)}{2a},
\]
which implies that $\pi(G_{uvw})\neq\pi(G_{vw})$, a contradiction.

\begin{table}[!h]
\caption{The triple $(n,q,L_v)$ in the proof of Lemma~\ref{lem28}}\label{tab2}
\centering
\begin{tabular}{|l|l|l|l|}
\hline
row & $n$ & $q$ & $L_v$ \\
\hline
1 & $3$ & $3$ & $\Sy_4$ \\
2 & $3$ & $9$ & $\PGL_2(9)$ \\
3 & $4$ & $3$ & $\Sy_6$ \\
4 & $4$ & odd & $\PSO_4^+(q).\C_2$ \\
5 & $8$ & odd & $\PSO_8^+(q).\C_2$ \\
\hline
\end{tabular}
\end{table}

Finally assume that~(iii) occurs. Then by Proposition~\ref{Homogeneous}, we deduce from the homogeneous factorization $L_v=L_{uv}L_{vw}$ that $\Soc(L_{uv})=\Soc(L_{vw})=\Soc(L_v)$, so
\[
(\Soc(L_v))^g=(\Soc(L_{uv}))^g=\Soc(L_{uv}^g)=\Soc(L_{vw})=\Soc(L_v),
\]
contradicting Lemma~\ref{lem26}. The proof is thus completed.
\end{proof}

We conclude this section with the following:

\begin{theorem}\label{thm4}
Let $G$ be an almost simple group with socle $\PSL_n(q)$. Then there is no $G$-vertex-primitive $(G,3)$-arc-transitive digraph $\Gamma$ such that $G_v$ is a maximal subgroup of $G$ from classes $\calC_7$ and $\calC_8$ for any vertex $v$ of $\Gamma$.
\end{theorem}

\begin{proof}
Suppose that $\Gamma$ is a $G$-vertex-primitive $(G,3)$-arc-transitive digraph such that $G_v$ is a maximal subgroup of $G$ from classes $\calC_7$ and $\calC_8$, where $v$ is a vertex of $\Gamma$. Then by Lemma~\ref{lem5}, $\Gamma$ has valency at least $3$. Hence Hypothesis~\ref{hyp2} holds. According to Lemmas~\ref{lem27} and~\ref{lem28}, $G_v$ cannot be a $\calC_i$-subgroup of $G$ for $i\in\{7,8\}$, a contradiction.
\end{proof}

We are now ready to prove the main theorem of the paper.

\begin{proof}[Proof of Theorem~$\ref{thm1}$]
Let $\Gamma$ be a $G$-vertex-primitive $(G,s)$-arc-transitive digraph with $s\geqslant3$, where $G$ is almost simple with socle $L=\PSL_n(q)$ and $q=p^f$ for some prime $p$, and let $v$ be a vertex of $\Gamma$. Then by Theorems~\ref{thm2},~\ref{thm3} and~\ref{thm4}, $G_v$ cannot be a $\calC_i$-subgroup of $G$ for $1\leqslant i\leqslant8$. If $G_v$ is a $\calC_9$-subgroup of $G$, then $G_v$ would be an almost simple group, which is not possible by Corollary~\ref{AS}. Thus we have $s\leqslant2$, as Theorem~$\ref{thm1}$ asserts.
\end{proof}

\vskip0.1in
\noindent\textsc{Acknowledgements.}~ This research was supported by Australian Research Council grant DP150101066. The third author's work on this paper was done when he was a research associate at the University of Western Australia. The second author acknowledges the support of NNSFC grant 11771200. The authors would like to thank the anonymous referee for careful reading and valuable suggestions.

\end{document}